\documentclass[a4paper,12pt]{article}

\usepackage[usenames,dvipsnames]{pstricks}

\usepackage{amscd,amssymb,amsmath,graphicx,verbatim,hyperref}
\usepackage[OT1,T1]{fontenc}

\usepackage[all, cmtip]{xy}
\usepackage{pdfpages}

\usepackage{amsmath,amsthm,amsfonts,mathabx,tikz,tikz-cd,stmaryrd,url,mathtools, mathrsfs}
\usepackage{hyperref}
\hypersetup{
	colorlinks,
	citecolor=black,
	filecolor=black,
	linkcolor=black,
	urlcolor=black
}

\usetikzlibrary{matrix,arrows,decorations.pathmorphing}

\newcommand{\missing}[1]{\marginpar{\color{red}missing:\\#1}}

\newcounter{Definitioncount}
\newtheorem{theorem}{Theorem}

\newtheorem{proposition}[theorem]{Proposition}
\newtheorem{corollary}[theorem]{Corollary}

\theoremstyle{definition}
\newtheorem*{Definition}{Definition}
\newtheorem*{Remark}{Remark}

\renewcommand{\theexample}{\arabic{example}}

\newtheoremstyle{fact}{\bigskipamount}{\medskipamount}{\upshape}{}{\itshape}{. }{ }{Fact}
\theoremstyle{fact}

\newtheoremstyle{genquest}{\bigskipamount}{\medskipamount}{\upshape}{}{\itshape}{. }{ }{General Question}
\theoremstyle{genquest}

\newtheoremstyle{step}{2\bigskipamount}{\medskipamount}{\upshape}{}{\itshape}{. }{ }{\underline{Step~\thestep}}
\theoremstyle{step}

\renewcommand{\thestep}{\arabic{step}}

\newcommand{\lra}{\longrightarrow}
\newcommand{\lla}{\longleftarrow}
\newcommand{\Lra}{\Longrightarrow}
\newcommand{\Lla}{\Longleftarrow}
\newcommand{\ldual}[1]{\mathord{{\let\nolimits\relax\sideset{^\wedge}{}{#1}}}}
\newcommand{\laction}[2]{\mathord{{\let\nolimits\relax\sideset{^{#1}}{}{#2}}}}
\newcommand{\conj}[2]{\mathord{{\let\nolimits\relax\sideset{^{#1}}{}{#2}}}}
\newcommand{\xylongto}[2][2pt]{{\UseTips{}\xymatrix @C+#1{ \ar[r]^{#2}&}}}
\newcommand{\ox}{\otimes}
\newcommand{\xra}{\xrightarrow}
\newcommand{\xla}{\xleftarrow}

\def\CA{{\mathscr A}}
\def\CB{{\mathscr B}}
\def\CC{{\mathscr C}}
\def\CD{{\mathscr D}}
\def\CE{{\mathscr E}}
\def\CF{{\mathscr F}}
\def\CG{{\mathscr G}}
\def\CH{{\mathscr H}}
\def\CI{{\mathscr I}}
\def\CJ{{\mathscr J}}
\def\CK{{\mathscr K}}
\def\CL{{\mathscr L}}
\def\CM{{\mathscr M}}
\def\CN{{\mathscr N}}
\def\CO{{\mathscr O}}
\def\CP{{\mathscr P}}
\def\CQ{{\mathscr Q}}
\def\CR{{\mathscr R}}
\def\CS{{\mathscr S}}
\def\CT{{\mathscr T}}
\def\CU{{\mathscr U}}
\def\CV{{\mathscr V}}
\def\CW{{\mathscr W}}
\def\CX{{\mathscr X}}
\def\CY{{\mathscr Y}}
\def\CZ{{\mathscr Z}}

\makeatletter
\newcommand*\bigcdot{\mathpalette\bigcdot@{.5}}
\newcommand*\bigcdot@[2]{\mathbin{\vcenter{\hbox{\scalebox{#2}{$\m@th#1\bullet$}}}}}
\makeatother

\def\theequation{{\thesection.\arabic{equation}}}

\begin{document}
\author{Ross Street \footnote{The author gratefully acknowledges the support of Australian Research Council Discovery Grant DP160101519.} \\ 
\small{Macquarie University, NSW 2109 Australia} \\
\small{<ross.street@mq.edu.au>}
}
\title{Vector product and composition algebras \\ in braided monoidal additive categories}
\date{\small{\today}}
\maketitle

\noindent {\small{\emph{2010 Mathematics Subject Classification:} 18D10; 15A03; 17A75; 11E20}
\\
{\small{\emph{Key words and phrases:} string diagram; vector product; bilinear form; braiding; monoidal dual.}}

\begin{abstract}
\noindent This is an account of some work of Markus Rost and his students Dominik Boos and Susanne Maurer. We adapt it to the braided monoidal setting.      
   
\end{abstract}

\section*{Introduction}

This is a fairly detailed account of some work of Markus Rost \cite{Rost1996} and his students Dominik Boos \cite{Boos1998} and Susanne Maurer \cite{Maur1998}. I am grateful to John Baez for alerting us to this material in a seminar at Macquarie University on 22 January 2003. The stimulus for revisiting this material was an invitation to
participate in the Workshop on Diagrammatic Reasoning in Higher Education during November 2018 in Newcastle, 
New South Wales\footnote{See \url{https://carma.newcastle.edu.au/meetings/drhe/}}.
It seemed a great opportunity to demonstrate the joy and power of string diagrams for
proving substantial algebraic facts.  

The contribution of the present paper is to adapt and generalize the ideas of Rost \cite{Rost1996} to the braided monoidal additive setting and to keep the diagrams
closer to those of Joyal and the author in \cite{37, 44}.
In this context, the goals are to prove that there are not many dimensions 
in which vector product algebras can exist and that the category of 
vector product algebras is equivalent to the category of composition algebras.

\section{Axiomatics}

\begin{Definition} A {\em vector product algebra (vpa)} over a commutative ring $R$ is an $R$-module $V$ equipped with a symmetric non-degenerate bilinear form $\bigcdot : V\ox V\to R$ and a linear map $\wedge : V\ox V\to V$ such that
\begin{eqnarray}
x\wedge y  & = & - y\wedge x \label{asym}\\
(x\wedge y)\bigcdot z  & = & (z\wedge x)\bigcdot y \label{cycsym}\\
(x \wedge y) \wedge z + x\wedge (y\wedge z)  & = & 2(x\bigcdot z)y - (x\bigcdot y)z - (y\bigcdot z) x \label{strange}
\end{eqnarray}
\end{Definition}

We call $\bigcdot$ the inner or dot product and $\wedge$ the exterior or vector product.
Condition \eqref{asym} expresses the antisymmetry of $\wedge$ and \eqref{cycsym} the cyclic symmetry; together they mean that $(x\wedge y)\bigcdot z$ is an alternating function of the variables.
Condition \eqref{strange} may seem strange or unfamiliar; however, we have the following observation.
\begin{proposition}
Assume $2= 1+1$ is cancellable in $R$. Then
\begin{itemize}
\item[(a)] \eqref{asym} is equivalent to \eqref{nilp}; 
\begin{eqnarray}
x\wedge x  & = & 0 \label{nilp}
\end{eqnarray}
\item[(b)] \eqref{cycsym} is equivalent to \eqref{scaltrip}; 
\begin{eqnarray}
(x\wedge y)\bigcdot z  & = & x\bigcdot (y\wedge z) \label{scaltrip}
\end{eqnarray}
\item[(c)] in the presence of \eqref{asym}, equation \eqref{strange} is equivalent to \eqref{famshort}.   
\begin{eqnarray}\label{famshort}
(x\wedge y)\wedge x  & = & (x\bigcdot x)y-(x\bigcdot y)x 
\end{eqnarray}
\end{itemize}
\end{proposition} 
\begin{proof} 
(a)  Put $x = y$ in \eqref{asym} to obtain $2 x\wedge x =0$;
then cancel the $2$. Conversely, apply \eqref{nilp} to $x+y$ in place of $x$, use linearity, and apply \eqref{nilp} twice more to obtain $x\wedge y + y\wedge x = 0$. 

\noindent (b) Using \eqref{cycsym} twice, we have
$(x\wedge y)\bigcdot z  = (z\wedge x)\bigcdot y = (y\wedge z)\bigcdot x = x\bigcdot (y\wedge z)$.  
Conversely, using \eqref{scaltrip} at the second step, $(z\wedge x)\bigcdot y  = y\bigcdot (z\wedge x) = (y\wedge z)\bigcdot x = x\bigcdot (y\wedge z)$. 

\noindent (c) Substitute $y=x$ in \eqref{strange} to obtain $0 + x\wedge (x\wedge z) = 2(x\bigcdot z)x-(x\bigcdot x)z-(x\bigcdot z)x = (x\bigcdot z)x-(x\bigcdot x)z$; but $x\wedge (x\wedge z) = - (x\wedge z)\wedge x$ by \eqref{asym}, yielding \eqref{famshort}. 
Conversely, replace $x$ by $x+z$ in \eqref{famshort} then $x$ by $x-z$ in \eqref{famshort} and subtract the two results. We obtain twice \eqref{strange}; cancel the $2$.  
\end{proof} 

\begin{Remark} In other words, for such a ring, \eqref{nilp}, \eqref{scaltrip}, \eqref{famshort} can be taken as alternative axioms for a vpa. Notice that properties \eqref{nilp} and \eqref{famshort} involve expressions in which terms have variables repeated. This causes a problem for defining the concept in a monoidal category since the diagonal function $V\to V\ox V$, $x\mapsto x\ox x$ is not linear. 
The reason \eqref{famshort} should seem more familiar than \eqref{strange} is that as undergraduates we learn the property
\begin{eqnarray}\label{fam}
(x\wedge y)\wedge z  & = & (x\bigcdot z)y-(x\bigcdot y)z 
\end{eqnarray}
for vectors in $\mathbb{R}^3$.
This condition \eqref{fam} does not have repeated variables.
Note that \eqref{famshort} is a special case of \eqref{fam}.
Moreover, \eqref{asym} and \eqref{fam} imply \eqref{strange} over any commutative ring.   
\end{Remark}

\begin{Definition} 
\begin{itemize}
\item[(com)] A vpa $V$ is called {\em commutative} when the wedge operation is zero: for all $x,y\in V$, $x\wedge y =0$.
\item[(ass)] A vpa $V$ is called {\em associative} when \eqref{fam} holds. 
\end{itemize}
\end{Definition}

\begin{Remark} The words ``commutative'' and ``associative'' do not mean the operations of $V$ have these properties in the usual sense. They apply to the following multiplication on $R\oplus V$:
\begin{eqnarray}
(\alpha, x)(\beta, y) = (\alpha \beta - x\bigcdot y, \alpha y + \beta x + x\wedge y) \ .\label{quat}
\end{eqnarray}
Indeed, \eqref{quat} is the formula for multiplication of quaternions when $V=\mathbb{R}^3$ with usual inner and vector products.   
\end{Remark}

\section{Monoidal concepts}

For a commutative ring $R$, the category $\mathrm{Mod}_R$ of $R$-modules (called $R$-vector spaces when $R$ is a field) and linear functions becomes a monoidal category with tensor product given by tensoring over $R$; so, up to isomorphism, $R$ is the unit object for tensoring.

Let $\CV$ denote any monoidal category in which the tensor product functor is denoted by $\ox : \CV\times \CV\to \CV$ and the unit object by $I$. 
Here is the notion of non-degeneracy for a morphism $A\ox B\to I$ in $\CV$.
A little later we will consider the notion of symmetry for such a morphism when $A=B$.  

\begin{Definition} A morphism $\varepsilon : A\ox B\to I$ called a {\em counit} for an adjunction (or duality) $A\dashv B$ when, for all objects $X$ and $Y$, the function 
\begin{eqnarray}
\CV(X,Y\ox A)\to \CV(X\ox B,Y) \ ,\label{radj}
\end{eqnarray} 
sending $X\xra{f} Y\ox A$ to the composite 
$X\ox B\xra{f\ox 1_B} Y\ox A\ox B\xra{1_Y\ox\varepsilon}Y$, is a bijection.
Taking $X=I$ and $Y=B$, we obtain a unique morphism $\eta : I\to B\ox A$, called the {\em unit} of the adjunction,
which maps to the identity of $B$ under the function \eqref{radj}.
\end{Definition}

\begin{proposition}
A morphism $\varepsilon : A\ox B\to I$ is a counit for an adjunction $A\dashv B$ if and only if there exists a morphism $\eta : I\to B\ox A$ satisfying the two equations depicted in \eqref{snakecon}.
\begin{eqnarray}\label{snakecon}
\begin{aligned}
\psscalebox{0.6 0.6} 
{
}
\end{aligned}
\end{eqnarray}
\end{proposition}
When there is no ambiguity, we denote counits by cups $\cup$ and units by caps $\cap$. So \eqref{radj} becomes the more geometrically ``obvious'' operation of pulling the ends of the strings as in \eqref{radjobv}.
These are sometimes called the {\em snake equations}.
\begin{eqnarray}\label{radjobv}
\begin{aligned}
\psscalebox{0.6 0.6} 
{
%
}
\end{aligned}
\end{eqnarray}

Now suppose the monoidal category is {\em braided}.
Then we have isomorphisms 
\begin{eqnarray}\label{braidingmorph}
c_{X,Y} : X\ox Y\lra Y\ox X
\end{eqnarray}
which we depict by a left-over-right crossing of strings in three dimensions;
the inverse is a right-over-left crossing.
\begin{eqnarray}
\begin{aligned}
\psscalebox{0.7 0.7} 
{
%
}
\end{aligned}
\end{eqnarray}
The braiding axioms reinforce the view that it behaves like a crossing. 
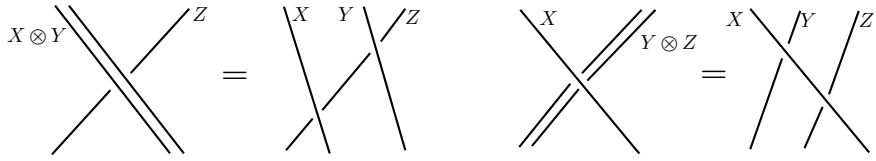
\begin{eqnarray*}
\begin{aligned}
\psscalebox{0.7 0.7} 
{
%
}
\end{aligned}
\end{eqnarray*}
The following Reidemeister move or Yang-Baxter equation is a consequence.
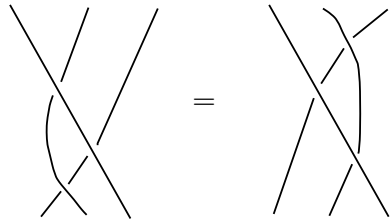
\begin{eqnarray*}
\begin{aligned}
\psscalebox{0.6 0.6} 
{
%
}
\end{aligned}
\end{eqnarray*}
We will refer to these properties as {\em the geometry of braiding} as justified by \cite{37}.
\begin{proposition}\label{symadj}
If $\CV$ is braided and $A\dashv B$ with counit and unit depicted by $\cup$ and $\cap$ then $B\dashv A$ with counit and unit depicted by
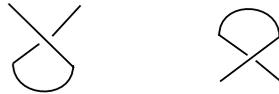
\begin{eqnarray*}
\begin{aligned}
\psscalebox{0.6 0.6} 
{
%
}
\end{aligned}
\end{eqnarray*} 
\end{proposition}
\begin{proof}
One of the snake equations is proved by the calculation
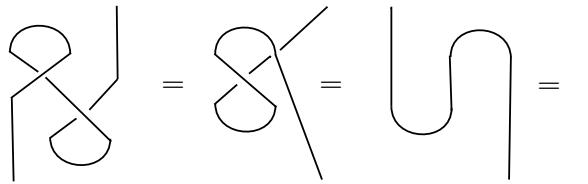
\begin{eqnarray*}
\begin{aligned}
\psscalebox{0.6 0.6} 
{
%
}
\end{aligned}
\end{eqnarray*} 
while the other is dual.
\end{proof}
\begin{Definition} Objects with duals have {\em dimension}:  if $A\dashv B$ then the dimension $d=d_A$ of $A$ is the following element of 
the commutative ring $\CV(I,I)$.   
\begin{eqnarray*}
\begin{aligned}
\psscalebox{0.6 0.6} 
{
%
}
\end{aligned}
\end{eqnarray*}
\end{Definition}
\begin{Definition} A morphism $g : A\ox A \to X$ is called {\em symmetric} when $$A\ox A \xra{c_{A,A}} A\ox A \xra{g} X \ \ = \ \ A\ox A \xra{g} X \ .$$
(Clearly the condition is equivalent to using the inverse braiding in place of the braiding.) 
In particular, a self-duality $A\dashv A$ with counit $\cup$ is called {\em symmetric} when  
\begin{eqnarray*}
\begin{aligned}
\psscalebox{0.6 0.6} 
{
%
}
\end{aligned}
\end{eqnarray*} 
By Proposition~\ref{symadj} and uniqueness of units, it follows that
\begin{eqnarray*}
\begin{aligned}
\psscalebox{0.6 0.6} 
{
%
}
\end{aligned}
\end{eqnarray*}
\end{Definition}
For a symmetric self-dual object $A$, the dimension is simply
\begin{eqnarray*}
\begin{aligned}
\psscalebox{0.6 0.6} 
{
%
}
\end{aligned}
\end{eqnarray*} 
\begin{proposition}
If $A\dashv A$ and $B\dashv B$ are symmetric self-dualities and $f : A \to B$ is a morphism then
\begin{eqnarray*}
\begin{aligned}
\psscalebox{0.6 0.6} 
{
%
}
\end{aligned}
\end{eqnarray*} 
\end{proposition}
\begin{proof}
Both sides are equal to:
\begin{eqnarray*}
\begin{aligned}
\psscalebox{0.6 0.6} 
{
%
}
\end{aligned}
\end{eqnarray*} 
\end{proof}
\begin{proposition}\label{Reide1}
If $A\dashv A$ is a symmetric self-duality then the following Reidemeister move holds
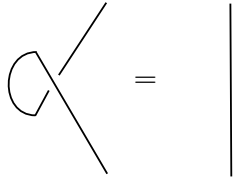
\begin{eqnarray*}
\begin{aligned}
\psscalebox{0.6 0.6} 
{
%
}
\end{aligned}
\end{eqnarray*} 
\end{proposition}
\begin{proof}
By dragging the bottom strings to the right and up over the top string we see that the proposition is the same as
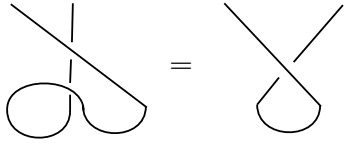
\begin{eqnarray*}
\begin{aligned}
\psscalebox{0.6 0.6} 
{
%
}\end{aligned}
\end{eqnarray*}
However, applying the snake equation and the geometry of braiding to the left-hand side, we are left with the expession for symmetry of the counit.  
\end{proof}
\begin{proposition}
If $\CV$ is a braided monoidal category then the set $\CV(I,I)$ of endomorphisms of the tensor unit $I$ is a commutative monoid under composition. In fact, composition agrees with tensor for such endomorphisms.
\end{proposition}
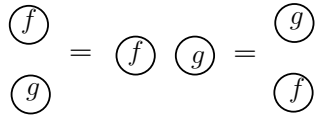
\begin{proof}
\begin{eqnarray*}
\begin{aligned}
\psscalebox{0.6 0.6} 
{
%
}
\end{aligned}
\end{eqnarray*}
\end{proof}
If $\CV$ is a braided monoidal {\em additive} category then there is an addition
in each hom set $\CV(X,Y)$ such that 
\begin{eqnarray}\label{distribconds}
\begin{aligned}
& h\circ 0\circ k = 0 \ , & h\circ (f+g)\circ k=h\circ f\circ k+h\circ g\circ k \ , 
\\ & U\ox 0\ox V =0 \ , & U\ox (f+g)\ox V =U\ox f\ox V+U\ox g\ox V \ .
\end{aligned}
\end{eqnarray}
When ambiguity is possible, we sometimes write $0_{X,Y}$ for the zero $0$ in the abelian group $\CV(X,Y)$.
In particular, $\CV(I,I)$ is a commutative ring; the multiplicative identity is depicted by an
empty string diagram $\varnothing$ as distinct from the zero $0_{II}$. 

\section{Scarcity of vector product algebras}

\begin{Definition} A {\em vector product algebra (vpa)} in a braided monoidal additive category $\CV$ is an object $V$ equipped with a symmetric self-duality $V\dashv V$ (depicted by a cup $\cup$) and a morphism $\wedge : V\ox V\to V$ (depicted by a Y)  such that the following three conditions hold.
\begin{eqnarray}\label{stringasym}
\begin{aligned}
\psscalebox{0.6 0.6} 
{

}
\end{aligned}
\end{eqnarray}
A vpa is {\em associative} when it satisfies
\begin{eqnarray}\label{stringass}
\begin{aligned}
\psscalebox{0.6 0.6} 
{
%
}
\end{aligned}
\end{eqnarray}
\end{Definition}

Using \eqref{stringasym} and \eqref{stringcycsym}, we see that \eqref{stringass} is equivalent to \eqref{stringassalt}.
\begin{eqnarray}\label{stringassalt}
\begin{aligned}
\psscalebox{0.6 0.6} 
{
%
}
\end{aligned}
\end{eqnarray}
By adding \eqref{stringass} and \eqref{stringassalt} we obtain \eqref{stringstrange}. So \eqref{stringstrange} is redundant in the definition of associative vpa.
\begin{proposition}\label{stringswing} 
The following is a consequence of \eqref{stringasym} and \eqref{stringcycsym}. 
\begin{eqnarray*}
\begin{aligned}
\psscalebox{0.6 0.6} 
{
%
}\end{aligned}
\end{eqnarray*}
\end{proposition}
\begin{proof}
Using \eqref{stringasym} and \eqref{stringcycsym} for the first equality below then the geometry of braiding for the second, we have
\begin{eqnarray*}
\begin{aligned}
\psscalebox{0.6 0.6} 
{
%
}
\end{aligned}
\end{eqnarray*}
However, the left-hand side is equal to the left-hand side of the equation in the proposition by \eqref{stringasym} while the right-hand sides are equal by symmetry of inner product $\cup$. 
\end{proof}
\begin{corollary}\label{stringswingcor1}
\begin{eqnarray*}
\begin{aligned}
\psscalebox{0.6 0.6} 
{
%
}
\end{aligned}
\end{eqnarray*}
and the snake equations \eqref{radjobv} yield the result.
\end{proof}
\begin{corollary}\label{stringswingcor2}
\begin{eqnarray*}
\begin{aligned}
\psscalebox{0.6 0.6} 
{
%
}
\end{aligned}
\end{eqnarray*}
\end{corollary}
\begin{proof}
Using Corollary~\ref{stringswingcor1}, we see that both sides are equal to
\begin{eqnarray*}
\begin{aligned}
\psscalebox{0.6 0.6} 
{
%
}
\end{aligned}
\end{eqnarray*}
and we have the result.
\end{proof}

\begin{corollary}\label{stringswingcor3}
\begin{eqnarray*}
\begin{aligned}
\psscalebox{0.6 0.6} 
{
%
}
\end{aligned}
\end{eqnarray*}
\end{corollary}
\begin{proof}
Attach a $\cap$ on the right input strings in all five terms of \eqref{stringass} and apply Corollary~\ref{stringswingcor1} to the two terms on the left-hand side.
\end{proof}
\begin{corollary}\label{stringswingcor4}
Condition \eqref{stringstrange} is equivalent to the following equation.
\begin{eqnarray*}
\begin{aligned}
\psscalebox{0.6 0.6} 
{
%
}
\end{aligned}
\end{eqnarray*}
\end{corollary}
\begin{proof}
Attach a $\cup$ on the right side of the output strings of all five terms of \eqref{stringstrange}. The right-hand side is what we want. On the left-hand side apply symmetry of $\bigcdot$ to the first term, and the alternating form axioms \eqref{stringasym} and \eqref{cycsym} to the second term. 
Each step can be reversed.
\end{proof}

\begin{theorem}
For any associative vector product algebra $V$ in any braided monoidal additive category $\CV$, the dimension $d=d_V$ satisfies the equation
\begin{eqnarray*}
d(d-1)(d-3)=0 
\end{eqnarray*}
in the endomorphism ring $\CV(I,I)$ of the tensor unit $I$.
\end{theorem}
\begin{proof}
We perform two string calculations each beginning with the following element of $\CV(I,I)$.
\begin{eqnarray}\label{mickey}
\begin{aligned}
\psscalebox{0.6 0.6} 
{

}
\end{aligned}
\end{eqnarray*}
in which, using Proposition~\ref{Reide1} and the geometry of braiding, each term reduced to a union of disjoint circles:
$$d - dd - dd + ddd = d(d-1)^2 \ .$$ 

Return now to \eqref{mickey} and apply Proposition~\ref{stringswing} to obtain
\begin{eqnarray*}
\begin{aligned}
\psscalebox{0.6 0.6} 
{
%
}
\end{aligned}
\end{eqnarray*}
in which we see we can apply \eqref{stringass} to obtain the following difference.
\begin{eqnarray*}
\begin{aligned}
\psscalebox{0.6 0.6} 
{
%
}
\end{aligned}
\end{eqnarray*}
In both terms we can apply \eqref{stringasym}.
\begin{eqnarray*}
\begin{aligned}
\psscalebox{0.6 0.6} 
{
%
}
\end{aligned}
\end{eqnarray*}
Now use geometry in the first term and Proposition~\ref{stringswing} in the second.
\begin{eqnarray*}
\begin{aligned}
\psscalebox{0.6 0.6} 
{
%
}
\end{aligned}
\end{eqnarray*}
Apply Proposition~\ref{Reide1} to the first term and the snake identity to the second to obtain
\begin{eqnarray*}
\begin{aligned}
\psscalebox{0.6 0.6} 
{
%
}
\end{aligned}
\end{eqnarray*}
where the first equality uses \eqref{stringasym} and the second uses Proposition~\ref{stringswing} on the first terms.
However, using Proposition~\ref{stringswing} and then \eqref{stringassalt}, we obtain 
\begin{eqnarray*}
\begin{aligned}
\psscalebox{0.6 0.6} 
{
%
}
\end{aligned}
\end{eqnarray*}
so, using Proposition~\ref{Reide1} and geometry, we again obtain terms which consist of one and two disjoint circles.
This shows that \eqref{mickey} is equal to $2(-d + d^2)$.
The two calculations therefore imply
$d(d -1)^2 = 2d(d-1)$. So $0=d(d-1)(d-1-2)=d(d-1)(d-3)$ as required.
 \end{proof}
 \begin{proposition}\label{evaluations}
Let $V$ be any vector product algebra in an additive braided monoidal category $\CV$ such that $2$ is cancellable in the abelian group $\CV(I,V)$. 
Then the following three equations hold.
\begin{eqnarray*}
\begin{aligned}
\psscalebox{0.6 0.6} 
{
\begin{pspicture}(0,-1.41)(20.22,1.41)
\psbezier[linecolor=black, linewidth=0.04](15.74007,0.3559871)(16.189953,1.6668495)(14.2524805,1.7574463)(14.61993,0.3840128965868189)
\psbezier[linecolor=black, linewidth=0.04](15.82007,-0.1440129)(15.717081,1.1868495)(14.622843,1.2174462)(14.53993,-0.11598710341318111)
\pscircle[linecolor=black, linewidth=0.04, dimen=outer](7.7,0.37){0.64}
\psline[linecolor=black, linewidth=0.04](8.3,0.53)(8.84,1.39)(8.86,1.39)
\psline[linecolor=black, linewidth=0.04](7.68,-0.27)(7.68,-1.37)(7.68,-1.39)
\rput[bl](8.96,0.15){\LARGE{$=$}}
\psline[linecolor=black, linewidth=0.04](11.82,1.39)(11.82,-1.41)
\psbezier[linecolor=black, linewidth=0.04](14.52,-0.05)(14.52,-0.85)(15.8,-0.89)(15.8,-0.09)
\rput[bl](16.18,0.23){\LARGE{$=$}}
\rput[bl](17.3,0.17){\Large{$(1-d)d$}}
\rput[bl](9.74,0.07){\Large{$(1-d)$}}
\pscircle[linecolor=black, linewidth=0.04, dimen=outer](0.64,0.39){0.64}
\rput[bl](2.76,0.27){\Large{$0$}}
\rput[bl](1.66,0.31){\LARGE{$=$}}
\psline[linecolor=black, linewidth=0.04](0.64,-0.25)(0.62,-1.13)
\end{pspicture}
}
\end{aligned}
\end{eqnarray*}
\end{proposition}
\begin{proof}
Using the asymmetry \eqref{stringasym} of $\wedge$ and the symmetry of $\bigcdot$, we obtain
\begin{eqnarray*}
\begin{aligned}
\psscalebox{0.6 0.6} 
{
\begin{pspicture}(0,-1.4033492)(7.28,1.4033492)
\pscircle[linecolor=black, linewidth=0.04, dimen=outer](1.26,0.71717244){0.52}
\psline[linecolor=black, linewidth=0.04](1.24,0.23717242)(1.2,-1.3428276)
\psbezier[linecolor=black, linewidth=0.04](4.7198744,0.7649961)(4.7363844,1.5648258)(3.5566359,1.5891783)(3.5401256,0.7893487284790911)
\psline[linecolor=black, linewidth=0.04](3.52,0.7971724)(4.5,-0.022827568)(4.1,-0.36282757)(4.1,-1.3828275)
\psline[linecolor=black, linewidth=0.04](4.72,0.77717245)(4.08,0.47717243)
\psline[linecolor=black, linewidth=0.04](3.84,0.37717244)(3.38,0.057172433)(4.1,-0.38282758)
\pscircle[linecolor=black, linewidth=0.04, dimen=outer](6.76,0.65717244){0.52}
\psline[linecolor=black, linewidth=0.04](6.74,0.17717244)(6.7,-1.4028276)
\rput[bl](5.0,-0.08282757){\LARGE{$=$}}
\rput[bl](2.08,-0.04282757){\LARGE{$=$}}
\rput[bl](0.0,-0.022827568){\Large{$-$}}
\end{pspicture}
}
\end{aligned}
\end{eqnarray*}
which proves the first equation after transposing and cancelling a $2$.

For the second equation, contract Corollary~\ref{stringswingcor3} on the left to obtain:
\begin{eqnarray*}
\begin{aligned}
\psscalebox{0.6 0.6} 
{
\begin{pspicture}(0,-1.8202301)(18.12,1.8202301)
\psbezier[linecolor=black, linewidth=0.04](1.9796251,-0.7139928)(2.0231707,0.084821194)(0.7839205,0.13234662)(0.74037486,-0.6664673689827475)
\psbezier[linecolor=black, linewidth=0.04](12.3,0.33976993)(12.3,-0.46023008)(13.54,-0.44023007)(13.54,0.3597699213027954)
\psline[linecolor=black, linewidth=0.04](1.9,-0.40023008)(2.36,0.3597699)(1.78,0.97976995)
\psline[linecolor=black, linewidth=0.04](2.36,0.37976992)(3.0,0.9997699)
\psbezier[linecolor=black, linewidth=0.04](6.079625,-0.03399279)(6.123171,0.7648212)(4.8839207,0.81234664)(4.840375,0.01353263101725247)
\psbezier[linecolor=black, linewidth=0.04](15.099626,-0.17399278)(15.14317,0.6248212)(13.90392,0.6723466)(13.860374,-0.12646736898274752)
\psline[linecolor=black, linewidth=0.04](6.06,0.0)(6.06,-0.70023006)
\psline[linecolor=black, linewidth=0.04](4.84,0.05976992)(4.86,-0.6802301)
\psline[linecolor=black, linewidth=0.04](5.98,0.33976993)(6.62,1.07977)
\psline[linecolor=black, linewidth=0.04](4.94,0.31976992)(4.34,0.97976995)
\psline[linecolor=black, linewidth=0.04](9.26,1.0597699)(10.6,-0.8002301)
\psline[linecolor=black, linewidth=0.04](10.58,1.0397699)(10.1,0.21976992)
\psline[linecolor=black, linewidth=0.04](9.86,-0.100230075)(9.28,-0.8002301)
\psline[linecolor=black, linewidth=0.04](13.54,0.27976993)(13.54,1.0597699)
\psline[linecolor=black, linewidth=0.04](12.3,0.2997699)(12.28,1.0397699)
\psline[linecolor=black, linewidth=0.04](13.86,-0.12023008)(13.86,-0.82023007)
\psline[linecolor=black, linewidth=0.04](15.08,-0.16023009)(15.08,-0.8402301)
\psline[linecolor=black, linewidth=0.04](17.0,1.0597699)(17.0,-0.82023007)
\psline[linecolor=black, linewidth=0.04](18.1,1.6797699)(18.1,-1.8202301)
\rput[bl](6.66,-0.02023008){\LARGE{$=$}}
\rput[bl](7.88,0.019769922){\Large{$2$}}
\rput[bl](2.86,-0.04023008){\Large{$+$}}
\rput[bl](15.56,-0.02023008){\Large{$-$}}
\rput[bl](10.68,0.019769922){\Large{$-$}}
\psbezier[linecolor=black, linewidth=0.04](0.02,-0.52023005)(0.02,-1.3202301)(0.74,-1.3402301)(0.74,-0.5402300786972046)
\psbezier[linecolor=black, linewidth=0.04](4.02,-0.64023006)(4.02,-1.4402301)(4.86,-1.4602301)(4.86,-0.6602300786972046)
\psbezier[linecolor=black, linewidth=0.04](8.58,-0.7402301)(8.58,-1.54023)(9.3,-1.5602301)(9.3,-0.7602300786972046)
\psbezier[linecolor=black, linewidth=0.04](11.66,-0.7402301)(11.66,-1.6202301)(13.88,-1.5602301)(13.88,-0.7602300786972046)
\psbezier[linecolor=black, linewidth=0.04](16.28,-0.6802301)(16.28,-1.4802301)(17.0,-1.5002301)(17.0,-0.7002300786972045)
\psbezier[linecolor=black, linewidth=0.04](17.000084,0.95642525)(16.959738,1.7554072)(16.29957,1.7420965)(16.339916,0.9431146039997009)
\psbezier[linecolor=black, linewidth=0.04](12.280085,0.9164252)(12.239738,1.7154073)(11.57957,1.7020966)(11.619916,0.9031146039997009)
\psbezier[linecolor=black, linewidth=0.04](9.372888,0.90114295)(8.931318,1.5699941)(8.559252,1.487248)(8.587112,0.6583968859710149)
\psbezier[linecolor=black, linewidth=0.04](4.4045243,0.8683374)(4.0390553,1.579978)(3.0900066,1.2028431)(3.4554756,0.49120243407398223)
\psbezier[linecolor=black, linewidth=0.04](1.7800844,0.9364252)(1.6749576,1.7354072)(-0.04521111,1.7220966)(0.059915606,0.9231146039997009)
\psline[linecolor=black, linewidth=0.04](3.44,0.49976993)(4.04,-0.70023006)(4.04,-0.70023006)
\psline[linecolor=black, linewidth=0.04](8.6,0.67976993)(8.58,-0.8002301)
\psline[linecolor=black, linewidth=0.04](11.6,0.9597699)(11.68,-0.8402301)
\psline[linecolor=black, linewidth=0.04](16.32,0.9597699)(16.28,-0.76023006)
\psline[linecolor=black, linewidth=0.04](0.04,0.9397699)(0.02,-0.6002301)
\psline[linecolor=black, linewidth=0.04](1.96,-0.6802301)(1.94,-1.42023)
\psline[linecolor=black, linewidth=0.04](6.06,-0.64023006)(6.06,-1.40023)
\psline[linecolor=black, linewidth=0.04](10.58,-0.76023006)(10.56,-1.54023)
\psline[linecolor=black, linewidth=0.04](15.08,-0.7802301)(15.04,-1.6402301)
\psline[linecolor=black, linewidth=0.04](2.98,0.97976995)(3.0,1.81977)
\psline[linecolor=black, linewidth=0.04](6.6,1.0597699)(6.6,1.7597699)
\psline[linecolor=black, linewidth=0.04](10.56,0.9997699)(10.54,1.71977)
\psline[linecolor=black, linewidth=0.04](13.54,1.0397699)(13.56,1.6797699)
\end{pspicture}
}
\end{aligned}
\end{eqnarray*}
in which the second term on the left is $0$ using Corollary~\ref{stringswingcor1}, the first term on the right is $2 \ 1_V$ by Proposition~\ref{Reide1}, the second term on the right is $1_V$ using the snake identities, and the final term is $d 1_V$. 
This proves the second equation of the proposition.

Using Corollary~\ref{stringswingcor1}, the left-hand side of the third equation is equal to the left-hand side of the equation:
\begin{eqnarray*}
\begin{aligned}
\psscalebox{0.6 0.6} 
{
\begin{pspicture}(0,-0.8844291)(6.32,0.8844291)
\psellipse[linecolor=black, linewidth=0.04, dimen=outer](0.5,-0.03183418)(0.5,0.66)
\psbezier[linecolor=black, linewidth=0.04](0.9908724,-0.17999357)(1.4827979,0.33411667)(1.321053,0.95043546)(0.8291276,0.43632520295189353)
\psbezier[linecolor=black, linewidth=0.04](3.760016,0.22744872)(3.7613502,1.0274476)(2.901318,1.0088818)(2.899984,0.20888291865225198)
\psbezier[linecolor=black, linewidth=0.04](3.28,-0.25183418)(3.28,-1.0518342)(4.36,-1.0318341)(4.36,-0.23183418273925782)
\psline[linecolor=black, linewidth=0.04](2.88,0.22816582)(3.28,-0.3118342)(3.76,0.24816582)
\psbezier[linecolor=black, linewidth=0.04](5.220016,-0.41255128)(5.22135,0.3874476)(4.361318,0.3688818)(4.359984,-0.431117081347748)
\psbezier[linecolor=black, linewidth=0.04](5.22,-0.2718342)(5.22,-1.0718342)(6.3,-1.0518342)(6.3,-0.25183418273925784)
\psbezier[linecolor=black, linewidth=0.04](6.2997704,0.283227)(6.28829,1.0831584)(3.7687492,1.033036)(3.7802296,0.2331046429836715)
\psline[linecolor=black, linewidth=0.04](6.28,0.32816583)(6.3,-0.39183417)
\rput[bl](1.98,-0.09183418){\LARGE{$=$}}
\end{pspicture}
}
\end{aligned}
\end{eqnarray*}
which is true by Proposition~\ref{stringswing}. The second equation of the proposition now gives the third equation.
\end{proof}
 \begin{proposition}[Tonny Albert Springer]\label{Springer}
For any vector product algebra $V$ with $2$ cancellable in $\CV(I,V)$,
\begin{eqnarray}
\begin{aligned}
\psscalebox{0.6 0.6} 
{
\begin{pspicture}(0,-1.9542617)(7.8361726,1.9542617)
\psline[linecolor=black, linewidth=0.04](0.018092956,1.9457383)(1.318093,-0.81426173)(1.318093,-1.9542618)
\psline[linecolor=black, linewidth=0.04](1.3380929,-0.8342617)(2.938093,1.9257382)
\psbezier[linecolor=black, linewidth=0.04](2.1178255,0.41490528)(2.1008785,1.1339314)(0.75092506,1.2555974)(0.7183603,0.45657123988963577)
\rput[bl](3.2780929,0.20573826){\LARGE{$=$}}
\rput[bl](4.238093,0.16573825){\Large{$(d-4)$}}
\psline[linecolor=black, linewidth=0.04](5.958093,1.7257383)(6.958093,-0.09426174)(6.978093,-1.8142618)
\psline[linecolor=black, linewidth=0.04](6.958093,-0.09426174)(7.818093,1.7257383)
\end{pspicture}
}
\end{aligned}
\end{eqnarray}
\end{proposition}
\begin{proof}
From \eqref{stringstrange}, we obtain
\begin{eqnarray}\label{stringstrangecontract1}
\begin{aligned}
\psscalebox{0.6 0.6} 
{
\begin{pspicture}(0,-1.5547639)(19.312122,1.5547639)
\psline[linecolor=black, linewidth=0.04](0.018221512,1.545635)(1.3782215,-1.454365)
\psline[linecolor=black, linewidth=0.04](0.8382215,-0.274365)(1.4782215,0.465635)
\psline[linecolor=black, linewidth=0.04](1.1182215,-0.894365)(2.2982216,0.465635)(2.2982216,0.48563498)
\psbezier[linecolor=black, linewidth=0.04](2.3353214,0.43886647)(2.257334,0.9353366)(1.383134,0.86887366)(1.4611216,0.37240352634627927)
\psbezier[linecolor=black, linewidth=0.04](5.8353214,0.71886647)(5.757334,1.2153366)(4.883134,1.1488737)(4.9611216,0.6524035263462793)
\psline[linecolor=black, linewidth=0.04](2.1582215,0.705635)(2.8382215,1.505635)
\psline[linecolor=black, linewidth=0.04](4.0782213,1.505635)(5.1182213,-1.454365)
\psline[linecolor=black, linewidth=0.04](4.7182217,-0.354365)(5.3182216,0.145635)(4.9382215,0.745635)
\psline[linecolor=black, linewidth=0.04](5.2782216,0.105634995)(5.8382215,0.705635)
\psline[linecolor=black, linewidth=0.04](5.6982217,0.945635)(6.2982216,1.525635)
\psbezier[linecolor=black, linewidth=0.04](11.215322,0.65886647)(11.137334,1.1553366)(10.263134,1.0888736)(10.341122,0.5924035263462792)
\psline[linecolor=black, linewidth=0.04](11.098222,0.825635)(11.698221,1.405635)
\psbezier[linecolor=black, linewidth=0.04](14.975322,0.59886646)(14.897334,1.0953366)(14.023134,1.0288737)(14.101122,0.5324035263462793)
\psline[linecolor=black, linewidth=0.04](14.858221,0.765635)(15.458221,1.3456349)
\psbezier[linecolor=black, linewidth=0.04](18.815321,0.6388665)(18.737333,1.1353366)(17.863134,1.0688736)(17.941122,0.5724035263462792)
\psline[linecolor=black, linewidth=0.04](18.698221,0.805635)(19.298222,1.385635)
\psline[linecolor=black, linewidth=0.04](10.338222,0.605635)(9.3782215,-1.414365)
\pscustom[linecolor=black, linewidth=0.04]
{
\newpath
\moveto(11.218222,0.685635)
\lineto(11.168221,0.26563498)
\curveto(11.143222,0.055634994)(11.088222,-0.254365)(11.058222,-0.354365)
\curveto(11.028221,-0.45436502)(10.963222,-0.579365)(10.928222,-0.604365)
\curveto(10.893222,-0.629365)(10.673222,-0.654365)(10.488221,-0.654365)
\curveto(10.303222,-0.654365)(10.093222,-0.564365)(10.018222,-0.29436502)
}
\psline[linecolor=black, linewidth=0.04](9.8782215,0.065634996)(9.118221,1.545635)
\psline[linecolor=black, linewidth=0.04](14.958221,0.605635)(14.918221,-1.474365)
\psbezier[linecolor=black, linewidth=0.04](13.298222,0.685635)(13.298222,-0.114365004)(14.118221,-0.114365004)(14.118221,0.685634994506836)
\psbezier[linecolor=black, linewidth=0.04](17.958221,0.585635)(17.998222,-0.754365)(18.798222,-0.654365)(18.798222,0.685634994506836)
\psline[linecolor=black, linewidth=0.04](17.398222,1.525635)(17.458221,-1.554365)
\psline[linecolor=black, linewidth=0.04](13.298222,0.62563497)(13.258222,1.465635)
\rput[bl](3.1182215,0.08563499){\LARGE{$+$}}
\rput[bl](8.398222,0.045634996){\Large{$2$}}
\rput[bl](12.218222,0.065634996){\Large{$-$}}
\rput[bl](6.6982217,0.045634996){\LARGE{$=$}}
\rput[bl](15.838222,0.065634996){\Large{$-$}}
\end{pspicture}
}
\end{aligned}
\end{eqnarray}
The first term on the left is the left-hand side $\ell$ (say) of the equation in the proposition.
Proposition~\ref{evaluations} applies to the second term on the left
yielding $(1-d) \ \wedge$.
Applying Corollary~\ref{stringswingcor1} to the first term on the right-hand side and using some geometry, we obtain the left-hand side of
\begin{eqnarray*}
\begin{aligned}
\psscalebox{0.6 0.6} 
{
\begin{pspicture}(0,-1.549658)(6.24,1.549658)
\pscustom[linecolor=black, linewidth=0.04]
{
\newpath
\moveto(1.26,-1.5456058)
\lineto(1.32,-1.2556058)
\curveto(1.35,-1.1106058)(1.425,-0.8856058)(1.47,-0.8056058)
\curveto(1.515,-0.7256058)(1.595,-0.5306058)(1.63,-0.4156058)
\curveto(1.665,-0.3006058)(1.74,-0.1256058)(1.78,-0.06560581)
\curveto(1.82,-0.005605812)(1.915,0.094394185)(1.97,0.13439418)
\curveto(2.025,0.17439419)(2.235,0.22439419)(2.39,0.2343942)
\curveto(2.545,0.24439418)(2.765,0.1543942)(2.83,0.05439419)
\curveto(2.895,-0.045605812)(2.96,-0.24560581)(2.96,-0.34560582)
\curveto(2.96,-0.4456058)(2.715,-0.5656058)(2.47,-0.5856058)
\curveto(2.225,-0.60560584)(1.95,-0.5756058)(1.92,-0.5256058)
\curveto(1.89,-0.47560582)(1.845,-0.3656058)(1.8,-0.18560581)
}
\psline[linecolor=black, linewidth=0.04](1.04,1.3343942)(1.4,-0.92560583)
\psline[linecolor=black, linewidth=0.04](1.68,0.07439419)(1.24,0.6943942)
\psline[linecolor=black, linewidth=0.04](0.94,0.8743942)(0.32,1.3343942)
\rput[bl](0.0,-0.06560581){\Large{$2$}}
\rput[bl](3.38,-0.025605813){\LARGE{$=$}}
\rput[bl](4.2,-0.045605812){\Large{$2$}}
\psline[linecolor=black, linewidth=0.04](4.98,0.5543942)(5.42,-0.5656058)
\psline[linecolor=black, linewidth=0.04](5.94,0.5543942)(5.42,-0.5856058)(5.44,-1.4856058)
\psline[linecolor=black, linewidth=0.04](5.0,0.5543942)(6.04,1.5143942)
\psline[linecolor=black, linewidth=0.04](5.92,0.5343942)(5.58,0.89439416)
\psline[linecolor=black, linewidth=0.04](5.32,1.0943942)(4.8,1.5343941)
\end{pspicture}
}
\end{aligned}
\end{eqnarray*}
which is equal to the right-hand side by Proposition~\ref{Reide1}, and this is equal to $-2 \ \wedge$ by \eqref{stringasym}.
The third term on the right-hand side of \eqref{stringstrangecontract1} is equal to $- \ \wedge$ by a snake equation. 
The fourth term of the right-hand side of \eqref{stringstrangecontract1} is $0$ as an easy consequence of the first identity of Proposition~\ref{evaluations}.

This leaves us with the equation
\begin{eqnarray*}
\ell + (1-d) \ \wedge = -2 \wedge \ - \ \wedge \ ,
\end{eqnarray*}
proving that $\ell = (d-4) \ \wedge$, as required.
\end{proof}
 \begin{proposition}\label{enabler}
For any vector product algebra $V$ and any morphism $$f : V\ox V\ox V\to X \ ,$$
the following identity holds.
\begin{eqnarray}
\begin{aligned}
\psscalebox{0.6 0.6} 
{
\begin{pspicture}(0,-2.1027567)(8.505272,2.1027567)
\pscircle[linecolor=black, linewidth=0.04, dimen=outer](1.0652722,-0.12226591){0.38}
\psline[linecolor=black, linewidth=0.04](1.0452722,-0.48226592)(1.0052723,-2.1022658)
\psbezier[linecolor=black, linewidth=0.04](2.0249014,1.4649657)(2.0031161,2.2646692)(0.003857733,2.2102058)(0.025643067,1.4105024208106727)
\psbezier[linecolor=black, linewidth=0.04](1.7228602,0.7720561)(1.6709899,1.5703727)(1.0158138,0.98172873)(1.0676842,0.1834120855402034)
\psline[linecolor=black, linewidth=0.04](2.0252721,1.4777341)(1.3652722,0.037734088)
\psline[linecolor=black, linewidth=0.04](0.025272217,1.4777341)(0.7452722,-0.022265911)
\rput[bl](3.1852722,-0.18226591){\LARGE{$=$}}
\rput[bl](6.5452724,-0.50226593){\Large{$f$}}
\rput[bl](0.9452722,-0.3422659){\Large{$f$}}
\pscircle[linecolor=black, linewidth=0.04, dimen=outer](6.705272,-0.2822659){0.4}
\pscustom[linecolor=black, linewidth=0.04]
{
\newpath
\moveto(7.0652723,-0.12226591)
\lineto(7.0252724,0.10773409)
\curveto(7.0052724,0.2227341)(6.955272,0.3877341)(6.925272,0.4377341)
\curveto(6.8952723,0.48773408)(6.8202724,0.5727341)(6.7752724,0.6077341)
\curveto(6.7302723,0.6427341)(6.575272,0.7127341)(6.4652724,0.74773407)
\curveto(6.3552723,0.7827341)(6.140272,0.8577341)(6.035272,0.8977341)
\curveto(5.930272,0.93773407)(5.7352724,1.012734)(5.6452723,1.0477341)
\curveto(5.555272,1.0827341)(5.345272,1.1777341)(5.225272,1.2377341)
\curveto(5.1052723,1.2977341)(4.9852724,1.4477341)(4.9852724,1.537734)
\curveto(4.9852724,1.6277341)(5.015272,1.787734)(5.0452724,1.8577341)
\curveto(5.075272,1.9277341)(5.385272,2.0227342)(5.665272,2.047734)
\curveto(5.9452724,2.072734)(6.5202723,2.082734)(6.8152723,2.067734)
\curveto(7.1102724,2.0527341)(7.515272,2.0027342)(7.6252723,1.9677341)
\curveto(7.7352724,1.9327341)(7.8802724,1.8627341)(7.915272,1.8277341)
\curveto(7.950272,1.7927341)(7.9852724,1.6177341)(7.9852724,1.4777341)
\curveto(7.9852724,1.3377341)(7.9602723,1.1327341)(7.935272,1.0677341)
\curveto(7.910272,1.0027341)(7.8552723,0.9077341)(7.825272,0.87773407)
\curveto(7.7952724,0.8477341)(7.705272,0.7877341)(7.6452723,0.75773406)
\curveto(7.5852723,0.7277341)(7.4802723,0.6677341)(7.435272,0.6377341)
\curveto(7.390272,0.6077341)(7.285272,0.5527341)(7.225272,0.5277341)
\curveto(7.165272,0.50273407)(7.100272,0.4727341)(7.0852723,0.45773408)
}
\psline[linecolor=black, linewidth=0.04](6.785272,0.35773408)(6.5052724,0.037734088)
\psbezier[linecolor=black, linewidth=0.04](6.0252185,1.0327525)(6.0425453,1.832565)(5.582653,1.882528)(5.565326,1.0827156190186216)
\pscustom[linecolor=black, linewidth=0.04]
{
\newpath
\moveto(6.0252724,0.7777341)
\lineto(6.055272,0.45773408)
\curveto(6.0702724,0.29773408)(6.1102724,0.07773409)(6.135272,0.017734088)
\curveto(6.160272,-0.04226591)(6.220272,-0.1372659)(6.325272,-0.24226591)
}
\psline[linecolor=black, linewidth=0.04](6.705272,-0.6622659)(6.705272,-2.002266)
\end{pspicture}
}
\end{aligned}
\end{eqnarray}
\end{proposition}
\begin{proof}
Applying Corollary~\ref{stringswingcor1} and geometry of braiding to the right-hand side yields the left-hand side of:
\begin{eqnarray*}
\begin{aligned}
\psscalebox{0.6 0.6} 
{
\begin{pspicture}(0,-2.1437285)(9.123686,2.1437285)
\pscircle[linecolor=black, linewidth=0.04, dimen=outer](1.6836863,-0.16323775){0.38}
\psline[linecolor=black, linewidth=0.04](1.6636863,-0.52323776)(1.6236862,-2.1432378)
\psbezier[linecolor=black, linewidth=0.04](0.84331536,1.543994)(0.8274143,2.3436973)(0.008155957,2.289234)(0.02405707,1.489530588535281)
\rput[bl](3.8036861,-0.22323774){\LARGE{$=$}}
\rput[bl](7.1636863,-0.54323775){\Large{$f$}}
\rput[bl](1.5636863,-0.38323775){\Large{$f$}}
\pscircle[linecolor=black, linewidth=0.04, dimen=outer](7.323686,-0.32323775){0.4}
\psbezier[linecolor=black, linewidth=0.04](7.1036325,0.63178074)(7.1209593,1.4315931)(6.661067,1.4815562)(6.64374,0.681743786743231)
\pscustom[linecolor=black, linewidth=0.04]
{
\newpath
\moveto(6.6436863,0.7367623)
\lineto(6.673686,0.41676226)
\curveto(6.6886864,0.25676227)(6.7286863,0.036762256)(6.7536864,-0.023237742)
\curveto(6.778686,-0.083237745)(6.838686,-0.17823774)(6.943686,-0.28323776)
}
\psline[linecolor=black, linewidth=0.04](7.323686,-0.7032378)(7.323686,-2.0432377)
\psline[linecolor=black, linewidth=0.04](0.023686219,1.5167623)(1.9836862,0.07676226)
\psline[linecolor=black, linewidth=0.04](0.82368624,1.6367623)(0.76368624,1.1767622)
\pscustom[linecolor=black, linewidth=0.04]
{
\newpath
\moveto(0.64368623,0.89676225)
\lineto(0.6736862,0.59676224)
\curveto(0.6886862,0.44676226)(0.75868624,0.25676227)(0.8136862,0.21676226)
\curveto(0.8686862,0.17676225)(1.0186862,0.10176226)(1.1136862,0.06676225)
\curveto(1.2086862,0.031762257)(1.3186862,-0.013237744)(1.3636862,-0.043237742)
}
\psline[linecolor=black, linewidth=0.04](0.7436862,0.27676225)(1.0236862,0.63676226)
\pscustom[linecolor=black, linewidth=0.04]
{
\newpath
\moveto(1.4436862,0.07676226)
\lineto(1.3836862,0.25676227)
\curveto(1.3536862,0.34676227)(1.3286862,0.44176227)(1.3436862,0.45676225)
}
\psbezier[linecolor=black, linewidth=0.04](1.4635518,0.69569904)(1.7857629,1.6308088)(1.4286793,1.4883851)(1.1438206,0.7978254717304936)
\psline[linecolor=black, linewidth=0.04](7.0836864,0.63676226)(7.243686,-0.0032377434)
\psbezier[linecolor=black, linewidth=0.04](7.603376,1.5084989)(7.587724,2.3083458)(6.3883452,2.2648726)(6.4039965,1.4650256018400878)
\psline[linecolor=black, linewidth=0.04](6.423686,1.4967623)(6.6636863,0.65676224)
\psline[linecolor=black, linewidth=0.04](7.5836864,1.5367622)(7.6036863,-0.083237745)
\end{pspicture}
}
\end{aligned}
\end{eqnarray*}
where the equality holds by Proposition~\ref{Reide1} and the geometry of braiding.
Now the result follows by Corollary~\ref{stringswingcor2}.
\end{proof}

 \begin{theorem}\label{scarcityvpa}
For any vector product algebra $V$ in any braided monoidal additive category $\CV$ such that $2$ can be cancelled in $\CV(I,V)$ and $\CV(I,I)$, the dimension $d=d_V$ satisfies the equation
\begin{eqnarray*}
d(d-1)(d-3)(d-7)=0 
\end{eqnarray*}
in the endomorphism ring $\CV(I,I)$ of the tensor unit $I$.
\end{theorem}
\begin{proof}
We perform two string calculations each beginning with the following element of $\CV(I,I)$.
\begin{eqnarray}\label{mounts}
\begin{aligned}
\psscalebox{0.6 0.6} 
{
\begin{pspicture}(0,-1.354533)(4.479997,1.354533)
\psbezier[linecolor=black, linewidth=0.04](0.039998475,-0.7495192)(0.039998475,-1.5495192)(4.4399986,-1.5095192)(4.4399986,-0.7095192337036133)
\psbezier[linecolor=black, linewidth=0.04](3.959987,-1.1630131)(3.9771585,0.99619853)(0.45718172,1.0031863)(0.3400101,-1.1160253286118313)
\psline[linecolor=black, linewidth=0.04](0.039998475,-0.7695192)(0.019998474,0.8104808)
\psline[linecolor=black, linewidth=0.04](4.4599986,-0.80951923)(4.4399986,0.77048075)
\psbezier[linecolor=black, linewidth=0.04](0.77993286,-0.12733188)(0.7752031,0.67265415)(0.035334285,0.6482794)(0.040064063,-0.15170658029628753)
\psbezier[linecolor=black, linewidth=0.04](4.439933,-0.10733189)(4.435203,0.69265413)(3.6353343,0.6682794)(3.640064,-0.13170658029628612)
\psbezier[linecolor=black, linewidth=0.04](4.4201727,0.71424246)(4.426821,1.5142171)(0.04647281,1.5466938)(0.03982431,0.7467190927957654)
\end{pspicture}
}
\end{aligned}
\end{eqnarray}

The first calculation involves noticing that the diagram of Proposition~\ref{Springer} occurs twice in \eqref{mounts}. Using that fact and the third equation of Proposition~\ref{evaluations}, we obtain the value
\begin{eqnarray}\label{value1}
(d-4)^2(1-d)d
\end{eqnarray}
for \eqref{mounts}.

The second calculation begins by considering the equation $\Xi$ obtained by mounting the following morphism on the top of each term of the equation of Corollary~\ref{stringswingcor4}.
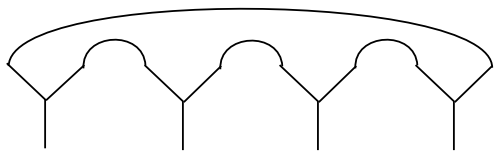
\begin{eqnarray*}
\begin{aligned}
\psscalebox{0.6 0.6} 
{
\begin{pspicture}(0,-1.5764492)(10.67379,1.5764492)
\psline[linecolor=black, linewidth=0.04](0.013790741,0.3435509)(0.85379076,-0.45644912)(0.85379076,-1.5364491)
\psline[linecolor=black, linewidth=0.04](0.87379074,-0.4764491)(1.7137908,0.3035509)
\psline[linecolor=black, linewidth=0.04](3.0337908,0.3035509)(3.8737907,-0.4964491)(3.8737907,-1.5764492)
\psline[linecolor=black, linewidth=0.04](3.8937907,-0.5164491)(4.713791,0.2835509)
\psline[linecolor=black, linewidth=0.04](5.9937906,0.3035509)(6.833791,-0.4964491)(6.833791,-1.5764492)
\psline[linecolor=black, linewidth=0.04](6.8537908,-0.5164491)(7.673791,0.2835509)
\psline[linecolor=black, linewidth=0.04](8.973791,0.3035509)(9.81379,-0.4964491)(9.81379,-1.5764492)
\psline[linecolor=black, linewidth=0.04](9.833791,-0.5164491)(10.65379,0.2835509)
\psbezier[linecolor=black, linewidth=0.04](3.0529337,0.2966608)(3.0089562,1.0955895)(1.6506702,1.0293697)(1.6946478,0.2304409632754141)
\psbezier[linecolor=black, linewidth=0.04](6.0529337,0.27666083)(6.008956,1.0755895)(4.65067,1.0093697)(4.694648,0.21044096327541412)
\psbezier[linecolor=black, linewidth=0.04](9.012934,0.2966608)(8.968956,1.0955895)(7.61067,1.0293697)(7.654648,0.2304409632754141)
\psbezier[linecolor=black, linewidth=0.04](10.65379,0.24532428)(10.640542,1.9164352)(0.11304308,2.0328884)(0.053791232,0.28177749313374817)
\end{pspicture}
}
\end{aligned}
\end{eqnarray*}
The reader should draw the diagram for equation $\Xi$. The first term on the left-hand side of $\Xi$ is then none other than the morphism \eqref{mounts} in the form of the left-hand side of Proposition~\ref{enabler} for an appropriate $f : V\ox V\ox V\to I$.
Notice also that the second term of the left-hand side of $\Xi$ is equal to the right hand side of Proposition~\ref{enabler} for the same $f$.
Consequently, the left-hand side of $\Xi$ is twice the value of \eqref{mounts}.  

Applying Proposition~\ref{gprop} to the first term of the right-hand side of $\Xi$ and using symmetry, we obtain twice the morphism \eqref{rhsfirstXi}.
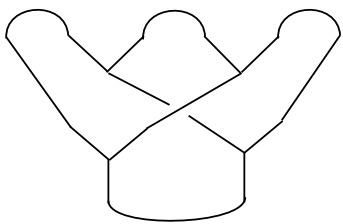
\begin{eqnarray}\label{rhsfirstXi}
\begin{aligned}
\psscalebox{0.6 0.6} 
{
\begin{pspicture}(0,-2.3631544)(7.3695617,2.3631544)
\psline[linecolor=black, linewidth=0.04](1.4304492,-0.25311303)(2.2704492,-0.97651726)(2.2704492,-1.9531131)
\psline[linecolor=black, linewidth=0.04](2.2904491,-0.9946024)(3.1304493,-0.28928325)
\psline[linecolor=black, linewidth=0.04](2.2304492,0.94688696)(3.0504491,1.746887)
\psline[linecolor=black, linewidth=0.04](5.190449,0.94688696)(6.0104494,1.746887)
\psbezier[linecolor=black, linewidth=0.04](1.3895922,1.7599969)(1.3456146,2.5589256)(-0.012671341,2.4927058)(0.031306252,1.6937770495486075)
\psbezier[linecolor=black, linewidth=0.04](4.389592,1.7399969)(4.3456144,2.5389256)(2.9873288,2.4727058)(3.0313063,1.6737770495486075)
\psbezier[linecolor=black, linewidth=0.04](7.349592,1.7599969)(7.3056145,2.5589256)(5.9473286,2.4927058)(5.9913063,1.6937770495486075)
\psline[linecolor=black, linewidth=0.04](1.3904492,1.766887)(2.2504492,0.966887)
\psline[linecolor=black, linewidth=0.04](4.370449,1.746887)(5.1704493,0.926887)
\psline[linecolor=black, linewidth=0.04](4.0304494,-0.013113022)(5.250449,-0.8301343)(5.250449,-1.933113)
\psline[linecolor=black, linewidth=0.04](5.250449,-0.83460236)(6.0904493,-0.12928323)
\psline[linecolor=black, linewidth=0.04](0.030449219,1.726887)(1.4504492,-0.273113)
\psline[linecolor=black, linewidth=0.04](3.1104493,-0.29311302)(5.190449,0.926887)
\psline[linecolor=black, linewidth=0.04](2.2704492,0.926887)(3.6104493,0.22688697)
\psline[linecolor=black, linewidth=0.04](7.350449,1.766887)(6.0704494,-0.11311302)
\psbezier[linecolor=black, linewidth=0.04](2.2704492,-1.893113)(2.2704492,-2.533113)(5.270449,-2.453113)(5.270449,-1.813113021850586)
\end{pspicture}
}
\end{aligned}
\end{eqnarray}
Applying vpa axioms leads to minus twice the morphism \eqref{rhsfirstXi2}
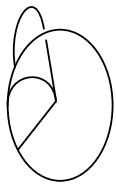
\begin{eqnarray}\label{rhsfirstXi2}
\begin{aligned}
\psscalebox{0.6 0.6} 
{
\begin{pspicture}(0,-1.9451683)(2.3616998,1.9451683)
\psbezier[linecolor=black, linewidth=0.04](1.1008428,-0.1920583)(1.0568652,0.6068704)(-0.0014207306,0.52065057)(0.042556863,-0.27827815675022066)
\psbezier[linecolor=black, linewidth=0.04](0.8378766,1.4223341)(0.6619431,2.4973342)(0.029589457,1.6423296)(0.20552301,0.5673294646401404)
\psellipse[linecolor=black, linewidth=0.04, dimen=outer](1.1916999,-0.25516823)(1.17,1.69)
\psline[linecolor=black, linewidth=0.04](0.8816998,1.1948317)(1.1016998,-0.14516823)
\psline[linecolor=black, linewidth=0.04](1.1016998,-0.16516823)(0.26169983,-1.2451682)
\end{pspicture}
}
\end{aligned}
\end{eqnarray}
in which we can recognize the left-hand side of Proposition~\ref{Springer}.
Applying that Springer Proposition and using vpa axioms, we find that the first term on the right-hand side of $\Xi$ is twice the morphism
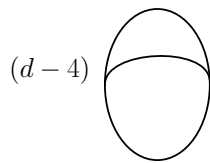
\begin{eqnarray}\label{rhsfirstXi3}
\begin{aligned}
\psscalebox{0.6 0.6} 
{
\begin{pspicture}(0,-1.69)(4.42,1.69)
\psbezier[linecolor=black, linewidth=0.04](4.3991427,0.08310993)(4.3551655,0.88203865)(2.0568795,0.77581877)(2.100857,-0.023109928600806597)
\psellipse[linecolor=black, linewidth=0.04, dimen=outer](3.25,0.0)(1.17,1.69)
\rput[bl](0.0,0.01){\Large{$(d-4)$}}
\end{pspicture}
}
\end{aligned}
\end{eqnarray}
By the third equation of Proposition~\ref{evaluations} we now have that the
first term of the right-hand side of $\Xi$ is
\begin{eqnarray*}
2(d-4)(1-d)d \ .
\end{eqnarray*}

Beginning with minus the second term of the right-hand side of $\Xi$ we have the calculation:
\begin{eqnarray}\label{rhssecondXi}
\begin{aligned}
\psscalebox{0.6 0.6} 
{
\begin{pspicture}(0,-1.71)(11.501543,1.71)
\psbezier[linecolor=black, linewidth=0.04](0.86068565,0.36310992)(0.816708,1.1620387)(-0.0015778962,0.7558188)(0.042399697,-0.0431099286008066)
\psellipse[linecolor=black, linewidth=0.04, dimen=outer](1.3615427,-0.02)(1.34,1.69)
\psbezier[linecolor=black, linewidth=0.04](1.6267942,0.4559943)(1.6027229,1.0790019)(0.8553735,0.89207655)(0.8562911,0.40400569814517323)
\psbezier[linecolor=black, linewidth=0.04](2.6681569,0.28754535)(2.4193864,0.99381787)(2.027479,1.1787271)(1.6549284,0.45245464437693045)
\psbezier[linecolor=black, linewidth=0.04](0.84154266,0.43)(0.84154266,-0.37)(1.6215427,-0.33)(1.6215427,0.47)
\psbezier[linecolor=black, linewidth=0.04](5.5406857,0.14310993)(5.489184,0.7060299)(4.530898,0.41981003)(4.5824,-0.1431099286008066)
\psellipse[linecolor=black, linewidth=0.04, dimen=outer](5.881543,0.0)(1.34,1.69)
\psbezier[linecolor=black, linewidth=0.04](6.986794,0.5159943)(6.962723,1.3862382)(4.8753734,1.0257642)(4.8762913,0.34400569814517323)
\psbezier[linecolor=black, linewidth=0.04](7.228157,0.067545354)(6.9793863,0.77381784)(6.667479,0.8987271)(6.2949286,0.17245464437693045)
\psbezier[linecolor=black, linewidth=0.04](5.5215425,0.21)(5.5215425,-0.59)(6.3015428,-0.55)(6.3015428,0.25)
\psellipse[linecolor=black, linewidth=0.04, dimen=outer](10.161543,0.02)(1.34,1.69)
\psbezier[linecolor=black, linewidth=0.04](11.481071,-0.13985495)(11.456562,0.6600038)(8.877504,0.51971364)(8.902014,-0.28014505088109215)
\psbezier[linecolor=black, linewidth=0.04](9.83744,0.37183994)(9.53284,1.1115824)(8.941046,0.8679025)(9.2456455,0.12816006231861593)
\rput[bl](7.8615427,0.01){\LARGE{$=$}}
\rput[bl](3.5015426,0.01){\LARGE{$=$}}
\end{pspicture}
}
\end{aligned}
\end{eqnarray}
and the second equation of Proposition~\ref{evaluations} applies to yield
\begin{eqnarray}\label{rhssecondXi2}
\begin{aligned}
\psscalebox{0.6 0.6} 
{
\begin{pspicture}(0,-1.69)(4.42,1.69)
\psbezier[linecolor=black, linewidth=0.04](4.3991427,0.08310993)(4.3551655,0.88203865)(2.0568795,0.77581877)(2.100857,-0.023109928600806597)
\psellipse[linecolor=black, linewidth=0.04, dimen=outer](3.25,0.0)(1.17,1.69)
\rput[bl](0.0,0.01){\Large{$(1-d)$}}
\end{pspicture}
}
\end{aligned}
\end{eqnarray}
resulting in the value of the second term of the right-hand side of $\Xi$ being
\begin{eqnarray}\label{lasttwoterms}
-(1-d)(1-d)d = -d(1-d)^2 \ . 
\end{eqnarray}

The third term on the right-hand side of $\Xi$ is minus the morphism
\begin{eqnarray}\label{rhsthirdXi}
\begin{aligned}
\psscalebox{0.6 0.6} 
{
\begin{pspicture}(0,-1.1884713)(15.738594,1.1884713)
\rput[bl](8.938604,-0.14847137){\LARGE{$=$}}
\rput[bl](4.118604,-0.12847137){\LARGE{$=$}}
\psellipse[linecolor=black, linewidth=0.04, dimen=outer](0.65860415,-0.52847135)(0.42,0.66)
\psellipse[linecolor=black, linewidth=0.04, dimen=outer](2.6786041,-0.52847135)(0.42,0.66)
\psbezier[linecolor=black, linewidth=0.04](2.295743,-0.42155778)(2.1985219,0.3725128)(0.9842443,0.27868563)(1.0814654,-0.5153849524877291)
\psbezier[linecolor=black, linewidth=0.04](3.3186238,0.4929208)(3.3413296,1.2925985)(0.06129011,1.3498142)(0.038584366,0.5501364509825379)
\psline[linecolor=black, linewidth=0.04](0.018604126,0.55152863)(0.3186041,-0.20847137)
\psline[linecolor=black, linewidth=0.04](3.3386042,0.55152863)(3.038604,-0.26847136)
\psellipse[linecolor=black, linewidth=0.04, dimen=outer](5.858604,-0.4884714)(0.42,0.66)
\psellipse[linecolor=black, linewidth=0.04, dimen=outer](7.878604,-0.4884714)(0.42,0.66)
\psbezier[linecolor=black, linewidth=0.04](8.155743,-0.041557796)(8.0214205,0.75251275)(5.5071425,0.8186856)(5.641465,0.024615047512270394)
\psbezier[linecolor=black, linewidth=0.04](8.518624,0.5329208)(8.541329,1.3325986)(5.26129,1.3898141)(5.2385845,0.5901364509825379)
\psline[linecolor=black, linewidth=0.04](5.218604,0.59152865)(5.5186043,-0.16847138)
\psline[linecolor=black, linewidth=0.04](8.538604,0.59152865)(8.238605,-0.22847137)
\rput[bl](9.598604,-0.10847137){\Large{$(1-d)(1-d)$}}
\psbezier[linecolor=black, linewidth=0.04](13.438604,-0.18847138)(13.438604,-0.9884714)(14.198605,-0.9884714)(14.198605,-0.18847137451171875)
\psbezier[linecolor=black, linewidth=0.04](14.938604,-0.12847137)(14.938604,-0.9284714)(15.698605,-0.9284714)(15.698605,-0.12847137451171875)
\psbezier[linecolor=black, linewidth=0.04](14.958158,-0.23005752)(14.919392,0.5690027)(14.160285,0.532175)(14.199051,-0.2668852193999186)
\psbezier[linecolor=black, linewidth=0.04](15.7185335,0.34426618)(15.709673,1.1442171)(13.409814,1.118742)(13.418674,0.31879108050100485)
\psline[linecolor=black, linewidth=0.04](13.418604,0.3515286)(13.438604,-0.30847138)
\psline[linecolor=black, linewidth=0.04](15.698605,0.37152863)(15.718604,-0.26847136)
\end{pspicture}
}
\end{aligned}
\end{eqnarray}
showing that we again obtain the value \eqref{lasttwoterms}.

Putting this all together in $\Xi$ and cancelling a 2, we obtain
\begin{eqnarray*}
(d-4)^2(1-d)d = (d-4)(1-d)d - d(1-d)^2 
\end{eqnarray*}
A little algebra turns this into:
\begin{align*}
0 & = d(d-1)((d-4)^2-(d-4)+(1-d)) \\
& = d(d-1)(d^2-10d+21)\\
& = d(d-1)(d-3)(d-7)
\end{align*}
as claimed.
\end{proof}

\section{Composition algebras}

\begin{Definition} A {\em composition algebra (ca)} in a braided monoidal additive category $\CV$ is an object $A$ equipped with a symmetric self-duality $A\dashv A$, a {\em multiplication} $m : A\ox A\to A$ and an {\em identity} $e : I\to A$ such that the following conditions hold.
\begin{eqnarray}\label{ca1}
\begin{aligned}
\psscalebox{0.6 0.6} 
{
\begin{pspicture}(0,-0.8182135)(8.54,0.8182135)
\psbezier[linecolor=black, linewidth=0.04](0.2999994,0.24298683)(0.33299413,-1.1379228)(1.8745677,-1.154153)(1.9400007,0.2534401748372085)
\pscircle[linecolor=black, linewidth=0.04, dimen=outer](1.94,0.5182135){0.28}
\pscircle[linecolor=black, linewidth=0.04, dimen=outer](0.28,0.5382135){0.28}
\rput[bl](0.15,0.37){\Large{$e$}}
\rput[bl](1.79,0.37){\Large{$e$}}
\rput[bl](3.22,0.018213501){\LARGE{$=$}}
\rput[bl](4.74,0.018213501){\LARGE{$\varnothing$}}
\end{pspicture}
}
\end{aligned}
\end{eqnarray}

\begin{eqnarray}\label{ca0}
\begin{aligned}
\psscalebox{0.6 0.6} 
{
\begin{pspicture}(0,-2.118585)(6.3102555,2.118585)
\pscircle[linecolor=black, linewidth=0.04, dimen=outer](0.75025576,-0.41858518){0.3}
\psline[linecolor=black, linewidth=0.04](0.010255737,2.101415)(1.1502557,1.4214149)(0.77025574,0.94141483)
\psline[linecolor=black, linewidth=0.04](1.9302558,2.0614147)(1.0102557,1.6614149)
\psline[linecolor=black, linewidth=0.04](0.69025576,1.5214149)(0.23025574,1.1414149)(1.1302557,0.46141484)(0.8702557,-0.17858517)
\psline[linecolor=black, linewidth=0.04](0.53025573,0.74141484)(0.17025574,0.32141483)(0.57025576,-0.19858517)
\psline[linecolor=black, linewidth=0.04](0.77025574,-0.69858515)(0.77025574,-2.118585)
\rput[bl](0.5502557,-0.55858517){\Large{$m$}}
\rput[bl](2.3102558,-0.23858517){\LARGE{$=$}}
\rput[bl](4.1502557,-0.57858515){\Large{$m$}}
\pscircle[linecolor=black, linewidth=0.04, dimen=outer](4.370256,-0.45858517){0.3}
\psline[linecolor=black, linewidth=0.04](4.390256,-0.69858515)(4.390256,-2.118585)
\psline[linecolor=black, linewidth=0.04](4.1702557,-0.29858518)(3.7902558,2.0614147)
\psline[linecolor=black, linewidth=0.04](4.5902557,-0.29858518)(5.1702557,2.0414147)
\end{pspicture}
}
\end{aligned}
\end{eqnarray}

\begin{eqnarray}\label{ca2}
\begin{aligned}
\psscalebox{0.6 0.6} 
{
\begin{pspicture}(0,-2.3240461)(14.04,2.3240461)
\pscircle[linecolor=black, linewidth=0.04, dimen=outer](0.32,0.93613815){0.32}
\pscircle[linecolor=black, linewidth=0.04, dimen=outer](1.68,-0.32386187){0.32}
\pscircle[linecolor=black, linewidth=0.04, dimen=outer](12.16,0.93613815){0.32}
\pscircle[linecolor=black, linewidth=0.04, dimen=outer](10.48,-0.28386188){0.32}
\psline[linecolor=black, linewidth=0.04](0.34,0.63613814)(1.48,-0.123861864)
\psline[linecolor=black, linewidth=0.04](1.9,-0.14386186)(2.96,2.316138)
\psline[linecolor=black, linewidth=0.04](1.68,-0.62386185)(1.68,-2.2838619)
\psline[linecolor=black, linewidth=0.04](6.18,2.316138)(6.22,-2.3238618)
\psline[linecolor=black, linewidth=0.04](12.18,0.61613816)(10.68,-0.06386187)
\psline[linecolor=black, linewidth=0.04](10.3,-0.083861865)(9.16,2.296138)
\psline[linecolor=black, linewidth=0.04](10.5,-0.60386187)(10.54,-2.3038619)
\rput[bl](7.86,-0.30386186){\LARGE{$=$}}
\rput[bl](10.26,-0.40386188){\Large{$m$}}
\rput[bl](0.16,0.81613815){\Large{$e$}}
\rput[bl](3.84,-0.26386186){\LARGE{$=$}}
\rput[bl](1.44,-0.44386187){\Large{$m$}}
\rput[bl](12.02,0.81613815){\Large{$e$}}
\end{pspicture}
}
\end{aligned}
\end{eqnarray}
\begin{eqnarray}\label{ca3}
\begin{aligned}
\psscalebox{0.6 0.6} 
{
\begin{pspicture}(0,-1.7862676)(17.319038,1.7862676)
\pscircle[linecolor=black, linewidth=0.04, dimen=outer](0.9990381,-0.87373227){0.32}
\pscircle[linecolor=black, linewidth=0.04, dimen=outer](2.999038,-0.8937323){0.32}
\pscircle[linecolor=black, linewidth=0.04, dimen=outer](9.659039,-0.87373227){0.32}
\pscircle[linecolor=black, linewidth=0.04, dimen=outer](7.699038,-0.8937323){0.32}
\psline[linecolor=black, linewidth=0.04](3.219038,-0.7137323)(4.279038,1.7462677)
\psline[linecolor=black, linewidth=0.04](7.519038,-0.69373226)(6.379038,1.6862677)
\rput[bl](11.559038,0.0062677194){\LARGE{$=$}}
\rput[bl](7.4790382,-1.0137323){\Large{$m$}}
\rput[bl](2.759038,-1.0137323){\Large{$m$}}
\psbezier[linecolor=black, linewidth=0.04](0.9990381,-1.1537323)(0.9990381,-1.9537323)(3.0190382,-1.9537323)(3.0190382,-1.1537322807312012)
\psbezier[linecolor=black, linewidth=0.04](7.739038,-1.1737323)(7.739038,-1.9737322)(9.659039,-1.9537323)(9.659039,-1.1537322807312012)
\psline[linecolor=black, linewidth=0.04](1.699038,1.7262677)(2.7990382,-0.7137323)
\psline[linecolor=black, linewidth=0.04](1.199038,-0.6537323)(2.0590382,0.6062677)
\psline[linecolor=black, linewidth=0.04](2.259038,0.8262677)(2.919038,1.7662678)(2.919038,1.7662678)
\psline[linecolor=black, linewidth=0.04](0.7990381,-0.6537323)(0.019038087,1.7662678)
\rput[bl](0.7790381,-0.9937323){\Large{$m$}}
\rput[bl](9.459038,-0.9937323){\Large{$m$}}
\psline[linecolor=black, linewidth=0.04](7.679038,1.6662678)(9.479038,-0.63373226)
\psline[linecolor=black, linewidth=0.04](8.899038,1.6462678)(9.879038,-0.69373226)
\psline[linecolor=black, linewidth=0.04](7.8990383,-0.6537323)(8.719038,0.08626772)(8.719038,0.08626772)
\psline[linecolor=black, linewidth=0.04](8.939038,0.26626772)(9.219038,0.5062677)
\psline[linecolor=black, linewidth=0.04](9.439038,0.6662677)(10.439038,1.6062677)
\rput[bl](5.219038,0.026267719){\LARGE{$+$}}
\rput[bl](12.839038,0.0062677194){\LARGE{$2$}}
\psbezier[linecolor=black, linewidth=0.04](14.019038,-0.81373227)(14.019038,-1.6137323)(15.379038,-1.5937322)(15.379038,-0.7937322807312012)
\psline[linecolor=black, linewidth=0.04](14.019038,-0.8537323)(14.019038,1.5062677)(13.999038,1.5262678)
\psline[linecolor=black, linewidth=0.04](15.379038,-0.8337323)(15.339038,1.5262678)
\psbezier[linecolor=black, linewidth=0.04](15.939038,-0.7737323)(15.939038,-1.5737323)(17.299038,-1.5537323)(17.299038,-0.7537322807312011)
\psline[linecolor=black, linewidth=0.04](17.299038,-0.7937323)(17.259039,1.5662677)
\psline[linecolor=black, linewidth=0.04](15.939038,-0.7937323)(15.859038,1.5262678)(15.859038,1.5462677)
\end{pspicture}
}
\end{aligned}
\end{eqnarray}
\end{Definition}

Assume that $\CV$ idempotents split in $\CV$. 
Assume also that multiplication by $2$ is invertible
in each hom abelian group $\CV(X,Y)$.

Define $\wedge : A\ox A\to A$ by the following equation where it is depicted by a {\em Y}-shaped string diagram. 
\begin{eqnarray}\label{2wedge}
\begin{aligned}
\psscalebox{0.5 0.5} 
{
\begin{pspicture}(0,-2.1044686)(11.8,2.1044686)
\pscircle[linecolor=black, linewidth=0.04, dimen=outer](6.24,-0.4044687){0.3}
\psline[linecolor=black, linewidth=0.04](6.26,-0.6844687)(6.26,-2.1044686)
\rput[bl](6.04,-0.5444687){\Large{$m$}}
\rput[bl](3.86,-0.16446869){\LARGE{$=$}}
\rput[bl](9.64,-0.5644687){\Large{$m$}}
\pscircle[linecolor=black, linewidth=0.04, dimen=outer](9.86,-0.44446868){0.3}
\psline[linecolor=black, linewidth=0.04](9.88,-0.6844687)(9.88,-2.1044686)
\psline[linecolor=black, linewidth=0.04](0.7,2.0755312)(1.78,0.05553131)(2.8,2.0955312)
\psline[linecolor=black, linewidth=0.04](1.78,0.015531311)(1.8,-2.0444686)
\rput[bl](7.92,-0.18446869){\LARGE{$-$}}
\rput[bl](0.0,-0.18446869){\LARGE{$2$}}
\psline[linecolor=black, linewidth=0.04](8.82,1.9955313)(10.52,0.73553133)(10.04,-0.2644687)
\psline[linecolor=black, linewidth=0.04](9.7,-0.24446869)(9.06,0.6555313)(9.64,1.1555313)
\psline[linecolor=black, linewidth=0.04](9.96,1.3755313)(10.7,1.9955313)
\psline[linecolor=black, linewidth=0.04](6.06,-0.2644687)(5.12,2.0755312)
\psline[linecolor=black, linewidth=0.04](6.46,-0.2244687)(7.14,2.0555313)
\end{pspicture}
}
\end{aligned}
\end{eqnarray} 

\begin{proposition}\label{cawedgeprop}
The following two equations hold.
\begin{eqnarray}\label{castringasym}
\begin{aligned}
\psscalebox{0.6 0.6} 
{
\begin{pspicture}(0,-1.1065843)(5.528201,1.1065843)
\psline[linecolor=black, linewidth=0.04](0.015518494,1.0939677)(0.7955185,0.1339677)(1.5755185,1.0539677)
\psline[linecolor=black, linewidth=0.04](0.7955185,0.1139677)(0.7955185,-1.0860323)
\psline[linecolor=black, linewidth=0.04](4.4555187,1.0139678)(5.4955187,0.15396771)(4.7755184,-0.3660323)(4.7555184,-1.1060323)
\psline[linecolor=black, linewidth=0.04](5.4555187,0.97396773)(5.1155186,0.6139677)
\psline[linecolor=black, linewidth=0.04](4.9355183,0.4339677)(4.4955187,0.033967704)(4.7755184,-0.3660323)
\rput[bl](2.1555185,-0.026032295){\LARGE{$=$}}
\rput[bl](3.5555184,-0.006032295){\LARGE{$-$}}
\end{pspicture}
}
\end{aligned}
\end{eqnarray}
\begin{eqnarray}\label{cathrowleg}
\begin{aligned}
\psscalebox{0.6 0.6} 
{
\begin{pspicture}(0,-0.8025911)(6.050782,0.8025911)
\psline[linecolor=black, linewidth=0.04](0.015518494,0.7899745)(0.7955185,-0.17002548)(1.5755185,0.7499745)
\rput[bl](2.7,-0.05002548){\LARGE{$=$}}
\psline[linecolor=black, linewidth=0.04](4.4755187,0.7699745)(5.2555184,-0.19002548)(6.0355186,0.7299745)
\psbezier[linecolor=black, linewidth=0.04](0.7955185,-0.090025485)(0.7955185,-0.8900255)(2.1955185,-0.9100255)(2.1955185,-0.11002548217773438)
\psbezier[linecolor=black, linewidth=0.04](3.8555186,-0.17002548)(3.8555186,-0.9700255)(5.2555184,-0.99002546)(5.2555184,-0.19002548217773438)
\psline[linecolor=black, linewidth=0.04](3.8555186,-0.17002548)(3.8355186,0.7499745)
\psline[linecolor=black, linewidth=0.04](2.1955185,-0.13002548)(2.1955185,0.7299745)
\end{pspicture}
}\end{aligned}
\end{eqnarray}
\end{proposition}
\begin{proof}
From \eqref{ca0}, we easily deduce \eqref{castringasym}.

By putting $e$ on the right-hand input strings in \eqref{ca3}, using the inverse braiding on the middle two strings, and applying \eqref{ca2}, we obtain the left equality in the following diagram.
\begin{eqnarray}\label{throwledproof}
\begin{aligned}
\psscalebox{0.6 0.6} 
{
\begin{pspicture}(0,-1.5406929)(20.47843,1.5406929)
\pscircle[linecolor=black, linewidth=0.04, dimen=outer](1.0284303,-0.61302817){0.35}
\psline[linecolor=black, linewidth=0.04](0.018430328,1.3969718)(0.77843034,-0.40302813)
\psline[linecolor=black, linewidth=0.04](2.0984304,1.4169718)(1.2984303,-0.42302814)
\psbezier[linecolor=black, linewidth=0.04](0.9984303,-0.94302815)(0.9984303,-1.7430282)(2.4384303,-1.6830281)(2.4384303,-0.8830281448364258)
\psline[linecolor=black, linewidth=0.04](2.4384303,-0.9230282)(2.8784304,1.3969718)
\pscircle[linecolor=black, linewidth=0.04, dimen=outer](5.9884305,-0.57302815){0.35}
\psbezier[linecolor=black, linewidth=0.04](4.59843,-0.94302815)(4.59843,-1.7430282)(6.03843,-1.6830281)(6.03843,-0.8830281448364258)
\psline[linecolor=black, linewidth=0.04](4.59843,-0.9830282)(4.1784306,1.3769718)
\psline[linecolor=black, linewidth=0.04](6.8784304,1.3569719)(5.4584303,0.43697184)(5.8584304,-0.30302814)
\psline[linecolor=black, linewidth=0.04](5.3984303,1.3169719)(5.9184303,1.0169718)
\psline[linecolor=black, linewidth=0.04](6.2984304,0.79697186)(6.8784304,0.43697184)(6.2184305,-0.34302816)
\psbezier[linecolor=black, linewidth=0.04](10.478431,-0.06302814)(10.478431,-0.59636146)(11.35843,-0.5563615)(11.35843,-0.02302814483642578)
\psbezier[linecolor=black, linewidth=0.04](10.03843,-0.38302815)(10.03843,-1.1830281)(12.03843,-1.1230282)(12.03843,-0.3230281448364258)
\pscircle[linecolor=black, linewidth=0.04, dimen=outer](11.348431,0.30697185){0.35}
\psline[linecolor=black, linewidth=0.04](10.478431,-0.083028145)(10.478431,1.3769718)
\psline[linecolor=black, linewidth=0.04](10.03843,-0.42302814)(9.95843,1.3569719)
\psline[linecolor=black, linewidth=0.04](12.03843,-0.34302816)(12.018431,1.3369719)
\rput[bl](7.6984305,-0.083028145){\LARGE{$=$}}
\rput[bl](9.19843,-0.083028145){\LARGE{$2$}}
\rput[bl](3.2984304,-0.14302814){\LARGE{$+$}}
\rput[bl](0.7584303,-0.7230281){\Large{$m$}}
\rput[bl](11.138431,0.17697185){\Large{$e$}}
\rput[bl](5.7184305,-0.68302816){\Large{$m$}}
\pscircle[linecolor=black, linewidth=0.04, dimen=outer](18.54843,-0.55302817){0.35}
\psline[linecolor=black, linewidth=0.04](17.59843,1.4769719)(18.35843,-0.32302815)
\psline[linecolor=black, linewidth=0.04](19.57843,1.4569719)(18.77843,-0.38302815)
\psbezier[linecolor=black, linewidth=0.04](17.21843,-0.94302815)(17.21843,-1.7430282)(18.65843,-1.6830281)(18.65843,-0.8830281448364258)
\psline[linecolor=black, linewidth=0.04](15.45843,-0.7830281)(15.89843,1.5369718)
\pscircle[linecolor=black, linewidth=0.04, dimen=outer](14.02843,-0.5330281){0.35}
\psbezier[linecolor=black, linewidth=0.04](14.018431,-0.8430281)(14.018431,-1.6430281)(15.45843,-1.5830282)(15.45843,-0.7830281448364258)
\rput[bl](16.31843,-0.083028145){\LARGE{$+$}}
\rput[bl](18.31843,-0.6630281){\Large{$m$}}
\rput[bl](13.7584305,-0.6630281){\Large{$m$}}
\psline[linecolor=black, linewidth=0.04](14.23843,-0.26302814)(14.87843,0.61697185)(13.49843,1.4569719)
\psline[linecolor=black, linewidth=0.04](13.83843,-0.22302814)(13.53843,0.59697187)(14.07843,0.9569719)
\psline[linecolor=black, linewidth=0.04](14.37843,1.1369718)(14.91843,1.4569719)
\psline[linecolor=black, linewidth=0.04](17.198431,-0.96302813)(16.79843,1.5169718)
\rput[bl](12.69843,-0.083028145){\LARGE{$=$}}
\end{pspicture}
}
\end{aligned}
\end{eqnarray}
The right-hand inequality is obtained similarly from \eqref{ca3} this time by putting $e$ on the third strings, putting the braiding on the first and last strings, then using \eqref{ca2} and the symmetry of the duality. Transposing terms in the equality of the left-hand side of
\eqref{throwledproof} with the right-hand side, we obtain twice equation \eqref{cathrowleg}.
\end{proof}

Applying \eqref{cathrowleg} of the Proposition and using the unit condition \eqref{ca2}, we obtain:

\begin{corollary}
\begin{eqnarray}\label{pwedge}
\begin{aligned}
\psscalebox{0.6 0.6} 
{
\begin{pspicture}(0,-1.4846326)(10.76,1.4846326)
\pscircle[linecolor=black, linewidth=0.04, dimen=outer](0.36,-0.52463263){0.36}
\rput[bl](0.22,-0.6446326){\Large{$e$}}
\pscircle[linecolor=black, linewidth=0.04, dimen=outer](8.9,-0.4246326){0.36}
\rput[bl](8.74,-0.5446326){\Large{$e$}}
\psbezier[linecolor=black, linewidth=0.04](0.38,-0.8646326)(0.38,-1.6646326)(1.68,-1.6646326)(1.68,-0.8646326065063477)
\psbezier[linecolor=black, linewidth=0.04](7.58,-0.7446326)(7.58,-1.5446326)(8.88,-1.5446326)(8.88,-0.7446326065063477)
\psline[linecolor=black, linewidth=0.04](0.8,1.4353673)(1.68,-0.8846326)(2.62,1.4553674)
\rput[bl](5.88,-0.2646326){\LARGE{$=$}}
\rput[bl](2.88,-0.2846326){\LARGE{$=$}}
\psline[linecolor=black, linewidth=0.04](6.42,1.4753674)(7.6,-0.7846326)(8.44,1.4753674)
\rput[bl](3.82,-0.4846326){\LARGE{$0_{A\ox A \ I}$}}
\end{pspicture}
}
\end{aligned}
\end{eqnarray} 
\end{corollary}

Notice that \eqref{ca1} says that $I\xra{e} A$ is a split monomorphism with left inverse
$p= (A\xra{1_A\ox e} A\ox A \xra{\bigcdot} I)$.
This gives the idempotent \eqref{idempforV} on $A$. 
\begin{eqnarray}\label{idempforV}
\begin{aligned}
\psscalebox{0.6 0.6} 
{
\begin{pspicture}(0,-1.4703069)(4.019998,1.4703069)
\pscircle[linecolor=black, linewidth=0.04, dimen=outer](2.2099981,-0.5400306){0.35}
\rput[bl](2.1,-0.6700306){\Large{$e$}}
\psbezier[linecolor=black, linewidth=0.04](2.239998,0.3299694)(2.239998,-0.20336393)(3.1199982,-0.16336393)(3.1199982,0.36996940612792967)
\psline[linecolor=black, linewidth=0.04](0.019998169,1.4099694)(0.05999817,-1.4700305)
\pscircle[linecolor=black, linewidth=0.04, dimen=outer](2.2099981,0.6399694){0.35}
\rput[bl](2.1,0.51){\Large{$e$}}
\psline[linecolor=black, linewidth=0.04](2.2199981,-0.8500306)(2.2199981,-1.4700305)
\psline[linecolor=black, linewidth=0.04](3.1199982,0.34996942)(3.0999982,1.4699694)
\rput[bl](0.85999817,-0.13003059){\LARGE{$-$}}
\end{pspicture}
}
\end{aligned}
\end{eqnarray} 
Splitting the idempotent, we obtain a direct sum diagram
\begin{eqnarray}\label{directsum}
\xymatrix @C+5mm{
 I \ar @<3pt> [r]^e  
 & A \ar @<3pt> [r]^q \ar @<3pt>[l]^p
 & V \ar @<3pt> [l]^i
}
\end{eqnarray}
exhibiting $A\cong I\oplus V$. 
This is expressed by the following equations.

\begin{eqnarray}\label{directsumstring1}
\begin{aligned}
\psscalebox{0.6 0.6} 
{
\begin{pspicture}(0,-2.0381572)(19.82,2.0381572)
\pscircle[linecolor=black, linewidth=0.04, dimen=outer](0.39,0.83720654){0.39}
\pscircle[linecolor=black, linewidth=0.04, dimen=outer](17.98,-1.0227935){0.36}
\pscircle[linecolor=black, linewidth=0.04, dimen=outer](9.52,0.77720654){0.36}
\pscircle[linecolor=black, linewidth=0.04, dimen=outer](0.36,-0.36279348){0.36}
\pscircle[linecolor=black, linewidth=0.04, dimen=outer](16.64,0.8572065){0.36}
\pscircle[linecolor=black, linewidth=0.04, dimen=outer](8.2,-0.92279345){0.36}
\psline[linecolor=black, linewidth=0.04](0.34,0.49720654)(0.34,-0.06279347)
\psline[linecolor=black, linewidth=0.04](0.36,-0.72279346)(0.38,-2.0227935)
\rput[bl](0.2,-0.5227935){\Large{$q$}}
\rput[bl](8.02,-1.0827935){\Large{$e$}}
\rput[bl](9.36,0.63720655){\Large{$i$}}
\rput[bl](3.14,-0.4427935){\LARGE{$0_{IV}$}}
\rput[bl](1.54,-0.36279348){\LARGE{$=$}}
\rput[bl](17.8,-1.1827935){\Large{$e$}}
\rput[bl](0.2,0.7572065){\Large{$e$}}
\rput[bl](9.68,1.4172065){\Large{$V$}}
\rput[bl](9.58,-1.4827935){\Large{$A$}}
\rput[bl](15.92,-1.4427935){\Large{$A$}}
\rput[bl](0.52,0.057206526){\Large{$A$}}
\rput[bl](15.78,1.4972066){\Large{$V$}}
\rput[bl](0.42,-1.6827935){\Large{$V$}}
\psbezier[linecolor=black, linewidth=0.04](8.18,-1.3027935)(8.18,-2.1027935)(9.48,-1.9827935)(9.48,-1.1827934741973878)
\psline[linecolor=black, linewidth=0.04](9.5,-1.2227935)(9.5,0.43720654)
\psline[linecolor=black, linewidth=0.04](9.52,1.0972065)(9.52,1.8972065)
\rput[bl](10.2,-0.22279347){\LARGE{$=$}}
\rput[bl](13.7,-0.22279347){\LARGE{$=$}}
\rput[bl](11.8,-0.32279348){\LARGE{$0_{VI}$}}
\psbezier[linecolor=black, linewidth=0.04](16.7,-1.4627935)(16.7,-2.2627935)(18.0,-2.1427934)(18.0,-1.3427934741973877)
\psline[linecolor=black, linewidth=0.04](16.72,-1.5027934)(16.66,0.49720654)
\psline[linecolor=black, linewidth=0.04](16.62,1.1972065)(16.58,2.0372064)
\rput[bl](16.5,0.7372065){\Large{$i$}}
\end{pspicture}
}
\end{aligned}
\end{eqnarray}

\begin{eqnarray}\label{directsumstring2}
\begin{aligned}
\psscalebox{0.6 0.6} 
{
\begin{pspicture}(0,-2.11)(16.7,2.11)
\pscircle[linecolor=black, linewidth=0.04, dimen=outer](0.83,0.89){0.39}
\pscircle[linecolor=black, linewidth=0.04, dimen=outer](4.54,-0.81){0.36}
\pscircle[linecolor=black, linewidth=0.04, dimen=outer](12.74,0.73){0.36}
\pscircle[linecolor=black, linewidth=0.04, dimen=outer](0.8,-0.31){0.36}
\pscircle[linecolor=black, linewidth=0.04, dimen=outer](12.76,-0.67){0.36}
\pscircle[linecolor=black, linewidth=0.04, dimen=outer](4.02,0.71){0.36}
\psline[linecolor=black, linewidth=0.04](0.8,0.55)(0.8,-0.01)
\psline[linecolor=black, linewidth=0.04](0.8,-0.67)(0.82,-1.97)
\rput[bl](0.66,0.77){\Large{$q$}}
\rput[bl](3.84,0.55){\Large{$e$}}
\rput[bl](0.68,-0.43){\Large{$i$}}
\rput[bl](4.38,-0.93){\Large{$e$}}
\rput[bl](0.0,0.11){\Large{$V$}}
\rput[bl](4.62,1.67){\Large{$A$}}
\rput[bl](1.02,1.69){\Large{$A$}}
\rput[bl](7.28,1.59){\Large{$A$}}
\rput[bl](12.12,1.49){\Large{$V$}}
\rput[bl](12.18,-1.69){\Large{$V$}}
\psbezier[linecolor=black, linewidth=0.04](4.0,0.33)(4.0,-0.47)(5.3,-0.35)(5.3,0.45)
\psline[linecolor=black, linewidth=0.04](5.32,0.41)(5.32,2.07)
\psline[linecolor=black, linewidth=0.04](12.74,1.05)(12.74,1.85)
\rput[bl](6.12,0.07){\LARGE{$=$}}
\rput[bl](13.64,-0.07){\LARGE{$=$}}
\rput[bl](12.6,0.59){\Large{$i$}}
\psline[linecolor=black, linewidth=0.04](0.78,1.27)(0.78,2.11)
\psline[linecolor=black, linewidth=0.04](4.54,-1.17)(4.54,-1.99)
\rput[bl](2.34,0.05){\LARGE{$+$}}
\rput[bl](0.94,-1.65){\Large{$A$}}
\rput[bl](4.7,-1.67){\Large{$A$}}
\psline[linecolor=black, linewidth=0.04](8.02,1.93)(8.02,-2.07)
\psline[linecolor=black, linewidth=0.04](12.76,0.39)(12.78,-0.37)
\psline[linecolor=black, linewidth=0.04](12.76,-1.03)(12.76,-2.09)
\psline[linecolor=black, linewidth=0.04](15.26,1.77)(15.26,-2.11)
\rput[bl](14.64,1.49){\Large{$V$}}
\rput[bl](12.64,-0.83){\Large{$q$}}
\rput[bl](12.16,-0.07){\Large{$A$}}
\end{pspicture}
}
\end{aligned}
\end{eqnarray}

An easy exercise using these equation shows:

\begin{proposition}\label{Vselfdual}
There is a self-duality $V\dashv V$ with counit and unit supplied by
 \begin{eqnarray*}
\begin{aligned}
\psscalebox{0.6 0.6} 
{
\begin{pspicture}(0,-1.0262643)(9.72,1.0262643)
\pscircle[linecolor=black, linewidth=0.04, dimen=outer](0.32,-0.09373573){0.32}
\pscircle[linecolor=black, linewidth=0.04, dimen=outer](2.42,-0.09373573){0.32}
\pscircle[linecolor=black, linewidth=0.04, dimen=outer](5.78,0.12626427){0.32}
\pscircle[linecolor=black, linewidth=0.04, dimen=outer](7.86,0.10626426){0.32}
\psbezier[linecolor=black, linewidth=0.04](0.34,-0.41373575)(0.34,-1.2137357)(2.42,-1.1937357)(2.42,-0.3937357330322266)
\psline[linecolor=black, linewidth=0.04](0.3,0.18626426)(0.3,1.0262643)
\psline[linecolor=black, linewidth=0.04](2.4,0.20626427)(2.38,1.0062642)
\psbezier[linecolor=black, linewidth=0.04](7.840035,0.3710176)(7.8546224,1.1709316)(5.814553,1.201425)(5.7999654,0.4015109375802308)
\psline[linecolor=black, linewidth=0.04](5.76,-0.19373573)(5.78,-0.8737357)
\psline[linecolor=black, linewidth=0.04](7.88,-0.17373574)(7.86,-0.91373575)
\rput[bl](2.32,-0.21373573){\Large{$i$}}
\rput[bl](7.68,-0.033735733){\Large{$q$}}
\rput[bl](2.54,0.6262643){\Large{$V$}}
\rput[bl](0.48,0.64626426){\Large{$V$}}
\rput[bl](5.32,-0.8337357){\Large{$V$}}
\rput[bl](2.4,-0.8937357){\Large{$A$}}
\rput[bl](5.52,0.7462643){\Large{$A$}}
\rput[bl](0.16,-0.23373574){\Large{$i$}}
\rput[bl](5.58,-0.013735733){\Large{$q$}}
\rput[bl](7.44,-0.8137357){\Large{$V$}}
\end{pspicture}
}
\end{aligned}
\end{eqnarray*}
With this choice of duality for $V$, 
 \begin{eqnarray*}
\begin{aligned}
\psscalebox{0.6 0.6} 
{
\begin{pspicture}(0,-1.7402761)(6.94,1.7402761)
\pscircle[linecolor=black, linewidth=0.04, dimen=outer](4.14,0.019969292){0.32}
\pscircle[linecolor=black, linewidth=0.04, dimen=outer](0.32,0.079969294){0.32}
\psbezier[linecolor=black, linewidth=0.04](2.92,-0.3000307)(2.92,-1.1000307)(4.12,-1.0600307)(4.12,-0.2600307083129883)
\rput[bl](0.5,1.2199693){\Large{$A$}}
\rput[bl](3.98,-0.100030705){\Large{$i$}}
\rput[bl](0.18,-0.06003071){\Large{$q$}}
\psline[linecolor=black, linewidth=0.04](0.34,0.3599693)(0.32,1.7399693)
\psline[linecolor=black, linewidth=0.04](0.36,-0.22003071)(0.34,-1.7400308)
\psbezier[linecolor=black, linewidth=0.04](5.4402494,0.34867033)(5.454152,1.1485496)(4.153653,1.1311474)(4.1397505,0.33126825549039635)
\psline[linecolor=black, linewidth=0.04](5.44,0.4199693)(5.42,-1.4600307)(5.42,-1.4600307)
\psline[linecolor=black, linewidth=0.04](2.92,-0.3600307)(2.88,1.7199693)
\rput[bl](1.66,0.0){\LARGE{$=$}}
\rput[bl](0.5,-1.4000307){\Large{$V$}}
\rput[bl](4.88,-1.4600307){\Large{$V$}}
\rput[bl](3.08,1.2399693){\Large{$A$}}
\end{pspicture}
}
\end{aligned}
\end{eqnarray*}
\end{proposition}

Notice that Corollary~\ref{pwedge} now tells us that $\wedge : A\ox A \to A$ 
factors through $i : V \to A$; indeed,
$$\wedge = e \ p \wedge + \ i \ q \wedge = 0 + i \ q \wedge = i \ q \wedge \ .$$
Risking ambiguity, we define a wedge for $V$ by 
\begin{equation}\label{wedgeforV}
\wedge : = (V\ox V \xra{i\ox i} A\ox A \xra{\wedge} A\xra{q}V)
\end{equation}
so that the following square commutes.
\begin{equation}\label{wedgerestricts}
\xymatrix{
V\ox V \ar[rr]^-{i\ox i} \ar[d]_-{\wedge} && A\ox A \ar[d]^-{\wedge} \\
V \ar[rr]_-{i} && A}
\end{equation}

\begin{proposition}\label{proponmandi}
\begin{itemize}
\item[(i)] 
\begin{eqnarray*}
\begin{aligned}
\psscalebox{0.6 0.6} 
{
\begin{pspicture}(0,-1.9047004)(12.857653,1.9047004)
\pscircle[linecolor=black, linewidth=0.04, dimen=outer](0.5376529,0.65529954){0.36}
\psbezier[linecolor=black, linewidth=0.04](1.3576529,-1.2247005)(1.3576529,-2.0247004)(2.6576529,-2.0247004)(2.6576529,-1.2247004318237305)
\rput[bl](8.997653,-0.044700433){\LARGE{$=$}}
\pscircle[linecolor=black, linewidth=0.04, dimen=outer](1.317653,-0.8847004){0.36}
\psline[linecolor=black, linewidth=0.04](0.51765287,0.95529956)(0.017652893,1.8952996)
\psline[linecolor=black, linewidth=0.04](0.6976529,0.35529956)(1.0976529,-0.6447004)
\psline[linecolor=black, linewidth=0.04](1.5376529,-0.6247004)(2.597653,1.8752996)
\psline[linecolor=black, linewidth=0.04](2.6576529,-1.2447004)(3.537653,1.8952996)
\rput[bl](4.2976527,-0.024700431){\LARGE{$+$}}
\pscircle[linecolor=black, linewidth=0.04, dimen=outer](6.137653,0.59529954){0.36}
\psbezier[linecolor=black, linewidth=0.04](6.957653,-1.2847004)(6.957653,-2.0847003)(8.257653,-2.0847003)(8.257653,-1.2847004318237305)
\pscircle[linecolor=black, linewidth=0.04, dimen=outer](6.917653,-0.9447004){0.36}
\psline[linecolor=black, linewidth=0.04](6.117653,0.89529955)(5.617653,1.8352996)
\psline[linecolor=black, linewidth=0.04](6.2976527,0.29529956)(6.697653,-0.7047004)
\rput[bl](10.657653,-0.024700431){\LARGE{$0$}}
\psline[linecolor=black, linewidth=0.04](7.437653,1.7752995)(8.197653,-0.18470043)(8.257653,-1.3247005)(8.237653,-1.3047004)
\psline[linecolor=black, linewidth=0.04](7.117653,-0.6447004)(7.717653,0.65529954)
\psline[linecolor=black, linewidth=0.04](7.937653,1.0152996)(8.317653,1.7752995)
\rput[bl](6.677653,-1.0447004){\Large{$m$}}
\rput[bl](6.037653,0.47529957){\Large{$i$}}
\rput[bl](1.0776529,-0.98470044){\Large{$m$}}
\rput[bl](0.4376529,0.5552996){\Large{$i$}}
\end{pspicture}
}
\end{aligned}
\end{eqnarray*}
\item[(ii)]
\begin{eqnarray*}
\begin{aligned}
\psscalebox{0.6 0.6} 
{
\begin{pspicture}(0,-1.8217549)(13.719689,1.8217549)
\pscircle[linecolor=black, linewidth=0.04, dimen=outer](2.8796897,0.67824507){0.36}
\psbezier[linecolor=black, linewidth=0.04](0.55968964,-1.161755)(0.55968964,-1.9617549)(1.8596896,-1.9617549)(1.8596896,-1.1617549514770509)
\rput[bl](10.179689,-0.021754952){\LARGE{$=$}}
\pscircle[linecolor=black, linewidth=0.04, dimen=outer](1.8596896,-0.841755){0.36}
\rput[bl](4.3796897,-0.0017549514){\LARGE{$+$}}
\pscircle[linecolor=black, linewidth=0.04, dimen=outer](8.879689,0.49824506){0.36}
\psbezier[linecolor=black, linewidth=0.04](6.1596894,-1.2017549)(6.1596894,-2.001755)(7.4596896,-2.001755)(7.4596896,-1.2017549514770507)
\pscircle[linecolor=black, linewidth=0.04, dimen=outer](7.4596896,-0.901755){0.36}
\rput[bl](11.51969,-0.0017549514){\LARGE{$0$}}
\rput[bl](7.21969,-1.001755){\Large{$m$}}
\rput[bl](8.759689,0.35824504){\Large{$i$}}
\rput[bl](1.6196896,-0.94175494){\Large{$m$}}
\rput[bl](2.7796896,0.538245){\Large{$i$}}
\psline[linecolor=black, linewidth=0.04](0.55968964,-1.2017549)(0.019689636,1.818245)
\psline[linecolor=black, linewidth=0.04](2.0996897,-0.5817549)(2.8396897,0.35824504)
\psline[linecolor=black, linewidth=0.04](2.9596896,0.99824506)(3.4196897,1.738245)
\psline[linecolor=black, linewidth=0.04](1.6796896,-0.60175496)(1.3196896,1.7782451)
\psline[linecolor=black, linewidth=0.04](7.29969,-0.6417549)(5.4396896,1.738245)
\psline[linecolor=black, linewidth=0.04](6.1596894,-1.241755)(6.1596894,-0.48175496)(6.4996896,0.03824505)
\psline[linecolor=black, linewidth=0.04](6.73969,0.45824504)(7.4796896,1.6782451)
\psline[linecolor=black, linewidth=0.04](7.6596894,-0.60175496)(8.759689,0.19824505)
\psline[linecolor=black, linewidth=0.04](8.97969,0.81824505)(9.53969,1.6382451)
\end{pspicture}
}
\end{aligned}
\end{eqnarray*}
\end{itemize}
\end{proposition}
\begin{proof}
For (i) put $i$ on the first and $e$ on the second input strings of axiom \eqref{ca3}
for a composition algebra. Then use \eqref{ca2} and \eqref{directsumstring1} to obtain the result.

For (ii) put $i$ on the last and $e$ on the third input strings of axiom \eqref{ca3}. Then use \eqref{ca2} and \eqref{directsumstring1} to obtain the result.
\end{proof}

\begin{proposition}\label{m+switchm}
 \begin{eqnarray*}
\begin{aligned}
\psscalebox{0.6 0.6} 
{
\begin{pspicture}(0,-2.0304296)(13.137447,2.0304296)
\pscircle[linecolor=black, linewidth=0.04, dimen=outer](4.617447,0.7300307){0.32}
\pscircle[linecolor=black, linewidth=0.04, dimen=outer](1.2574469,-0.44996932){0.32}
\psbezier[linecolor=black, linewidth=0.04](10.137447,0.45003068)(10.137447,-0.3499693)(11.337447,-0.3099693)(11.337447,0.49003068923950194)
\rput[bl](11.217447,0.6100307){\Large{$i$}}
\rput[bl](7.677447,-0.08996931){\LARGE{$=$}}
\pscircle[linecolor=black, linewidth=0.04, dimen=outer](2.277447,0.7500307){0.32}
\pscircle[linecolor=black, linewidth=0.04, dimen=outer](0.4774469,0.7300307){0.32}
\pscircle[linecolor=black, linewidth=0.04, dimen=outer](6.577447,0.7100307){0.32}
\pscircle[linecolor=black, linewidth=0.04, dimen=outer](5.537447,-0.8099693){0.32}
\pscircle[linecolor=black, linewidth=0.04, dimen=outer](10.757447,-0.7099693){0.32}
\pscircle[linecolor=black, linewidth=0.04, dimen=outer](10.137447,0.7300307){0.32}
\pscircle[linecolor=black, linewidth=0.04, dimen=outer](11.337447,0.7700307){0.32}
\psline[linecolor=black, linewidth=0.04](10.117447,1.0100307)(10.097446,2.0300307)
\psline[linecolor=black, linewidth=0.04](11.357447,1.0300307)(11.337447,2.0300307)
\psline[linecolor=black, linewidth=0.04](10.757447,-1.0299693)(10.757447,-1.9899693)
\psline[linecolor=black, linewidth=0.04](0.4774469,1.0300307)(0.0174469,1.8500307)
\psline[linecolor=black, linewidth=0.04](2.337447,1.0500307)(2.9374468,1.8100307)
\psline[linecolor=black, linewidth=0.04](0.4774469,0.4300307)(1.0374469,-0.22996931)
\psline[linecolor=black, linewidth=0.04](2.2174468,0.4700307)(1.4974469,-0.2699693)
\psline[linecolor=black, linewidth=0.04](1.2574469,-0.7699693)(1.2574469,-1.9699693)
\psline[linecolor=black, linewidth=0.04](4.597447,1.0300307)(4.177447,1.8300307)
\psline[linecolor=black, linewidth=0.04](6.677447,0.97003067)(7.197447,1.8100307)
\psline[linecolor=black, linewidth=0.04](5.517447,-1.1299694)(5.537447,-2.0299692)
\psline[linecolor=black, linewidth=0.04](4.677447,0.4100307)(6.477447,-0.06996931)(5.717447,-0.5699693)
\psline[linecolor=black, linewidth=0.04](6.477447,0.45003068)(5.8374467,0.23003069)
\psline[linecolor=black, linewidth=0.04](5.417447,0.05003069)(4.737447,-0.22996931)
\psline[linecolor=black, linewidth=0.04](4.717447,-0.22996931)(5.2974467,-0.58996934)
\rput[bl](10.0174465,0.6100307){\Large{$i$}}
\rput[bl](6.417447,0.5700307){\Large{$i$}}
\rput[bl](4.477447,0.6100307){\Large{$i$}}
\rput[bl](2.117447,0.6100307){\Large{$i$}}
\rput[bl](0.3374469,0.59003067){\Large{$i$}}
\rput[bl](5.2974467,-0.96996933){\Large{$m$}}
\rput[bl](10.597446,-0.8299693){\Large{$e$}}
\rput[bl](1.0174469,-0.6099693){\Large{$m$}}
\rput[bl](3.3974469,-0.08996931){\LARGE{$+$}}
\rput[bl](8.697447,-0.1699693){\LARGE{$-2$}}
\end{pspicture}
}
\end{aligned}
\end{eqnarray*}
\end{proposition}
\begin{proof}
Let us temporarily write $n : V\ox V\to A$ for the left-hand side of the equation.
Using Proposition~\ref{proponmandi} and the counit symmetry, 
we have the following calculation. The first step uses Proposition~\ref{proponmandi} (i)
and symmetry, the second step uses Proposition~\ref{proponmandi} (ii) and symmetry,
while the last step uses Proposition~\ref{proponmandi} (i) again.  
 \begin{eqnarray*}
\begin{aligned}
\psscalebox{0.6 0.6} 
{
\begin{pspicture}(0,-1.6982949)(21.39731,1.6982949)
\pscircle[linecolor=black, linewidth=0.04, dimen=outer](0.5973105,0.4370378){0.34}
\pscircle[linecolor=black, linewidth=0.04, dimen=outer](1.9973105,0.3970378){0.34}
\pscircle[linecolor=black, linewidth=0.04, dimen=outer](3.3373106,0.4170378){0.34}
\pscircle[linecolor=black, linewidth=0.04, dimen=outer](1.2773105,-0.5829622){0.34}
\psbezier[linecolor=black, linewidth=0.04](1.2973105,-0.9029622)(1.2973105,-1.7029622)(3.2973106,-1.6429622)(3.2973106,-0.8429621982574463)
\psline[linecolor=black, linewidth=0.04](3.2973106,-0.9029622)(3.3173106,0.0970378)
\psline[linecolor=black, linewidth=0.04](1.4773105,-0.3429622)(2.0373106,0.0970378)
\psline[linecolor=black, linewidth=0.04](1.0373105,-0.3629622)(0.5773105,0.1170378)
\psline[linecolor=black, linewidth=0.04](0.55731046,0.7370378)(0.037310485,1.5170377)(0.037310485,1.5170377)
\psline[linecolor=black, linewidth=0.04](3.3573105,0.7370378)(4.0173106,1.5370378)
\psline[linecolor=black, linewidth=0.04](1.9973105,0.7170378)(1.9773105,1.5370378)
\rput[bl](15.43731,-0.022962198){\LARGE{$=$}}
\rput[bl](19.47731,0.8370378){\Large{$i$}}
\rput[bl](1.0773104,-0.7229622){\Large{$m$}}
\rput[bl](4.3773103,-0.1829622){\LARGE{$=$}}
\rput[bl](17.91731,-0.8429622){\Large{$m$}}
\rput[bl](3.1573105,0.2970378){\Large{$i$}}
\rput[bl](1.8573105,0.2770378){\Large{$i$}}
\rput[bl](0.4373105,0.2970378){\Large{$i$}}
\pscircle[linecolor=black, linewidth=0.04, dimen=outer](8.077311,-0.6629622){0.34}
\pscircle[linecolor=black, linewidth=0.04, dimen=outer](8.05731,0.5770378){0.34}
\pscircle[linecolor=black, linewidth=0.04, dimen=outer](9.417311,0.6170378){0.34}
\pscircle[linecolor=black, linewidth=0.04, dimen=outer](6.6973104,0.5770378){0.34}
\pscircle[linecolor=black, linewidth=0.04, dimen=outer](19.65731,0.9770378){0.34}
\psbezier[linecolor=black, linewidth=0.04](6.7373104,-0.8629622)(6.7373104,-1.6629622)(8.09731,-1.7829622)(8.09731,-0.9829621982574462)
\psline[linecolor=black, linewidth=0.04](6.7373104,-0.9229622)(6.6773105,0.23703781)
\psline[linecolor=black, linewidth=0.04](6.6973104,0.8970378)(8.037311,1.5970378)(8.077311,1.5770378)
\psline[linecolor=black, linewidth=0.04](7.9773107,0.8770378)(7.4173107,1.1370378)
\psline[linecolor=black, linewidth=0.04](7.0773106,1.2370378)(6.0973105,1.5570378)
\psline[linecolor=black, linewidth=0.04](7.8773103,-0.4429622)(8.09731,0.2570378)
\psline[linecolor=black, linewidth=0.04](8.297311,-0.4429622)(9.417311,0.2970378)
\psline[linecolor=black, linewidth=0.04](9.43731,0.9370378)(9.85731,1.5970378)
\rput[bl](16.45731,-0.0029621983){\LARGE{$-$}}
\rput[bl](7.8773103,-0.7829622){\Large{$m$}}
\rput[bl](5.3973103,-0.1229622){\LARGE{$-$}}
\rput[bl](9.23731,0.4770378){\Large{$i$}}
\rput[bl](7.8773103,0.4570378){\Large{$i$}}
\rput[bl](6.4973106,0.4370378){\Large{$i$}}
\pscircle[linecolor=black, linewidth=0.04, dimen=outer](14.257311,0.9970378){0.34}
\pscircle[linecolor=black, linewidth=0.04, dimen=outer](13.05731,0.077037804){0.34}
\pscircle[linecolor=black, linewidth=0.04, dimen=outer](12.05731,1.0370378){0.34}
\pscircle[linecolor=black, linewidth=0.04, dimen=outer](13.297311,-0.7029622){0.34}
\psbezier[linecolor=black, linewidth=0.04](13.31731,-1.0029622)(13.31731,-1.8029622)(14.677311,-1.9229622)(14.677311,-1.1229621982574463)
\psline[linecolor=black, linewidth=0.04](14.677311,-1.1829622)(14.63731,0.2970378)(12.217311,0.7570378)(12.19731,0.7570378)
\psline[linecolor=black, linewidth=0.04](13.077311,-0.5229622)(13.05731,-0.2429622)
\psline[linecolor=black, linewidth=0.04](13.01731,0.3770378)(12.997311,0.5170378)
\psline[linecolor=black, linewidth=0.04](12.997311,0.7570378)(12.93731,1.6970378)
\psline[linecolor=black, linewidth=0.04](11.837311,1.2770379)(11.377311,1.6170378)
\psline[linecolor=black, linewidth=0.04](13.497311,-0.4629622)(14.05731,0.2770378)
\psline[linecolor=black, linewidth=0.04](14.23731,0.6770378)(14.1573105,0.4970378)
\psline[linecolor=black, linewidth=0.04](14.31731,1.2970378)(14.797311,1.6370378)
\rput[bl](13.117311,-0.8229622){\Large{$m$}}
\rput[bl](14.05731,0.8570378){\Large{$i$}}
\rput[bl](12.837311,-0.042962197){\Large{$i$}}
\rput[bl](11.85731,0.9170378){\Large{$i$}}
\rput[bl](10.757311,-0.102962196){\LARGE{$=$}}
\pscircle[linecolor=black, linewidth=0.04, dimen=outer](18.49731,0.9770378){0.34}
\pscircle[linecolor=black, linewidth=0.04, dimen=outer](16.99731,0.9770378){0.34}
\pscircle[linecolor=black, linewidth=0.04, dimen=outer](18.13731,-0.7229622){0.34}
\psbezier[linecolor=black, linewidth=0.04](18.15731,-1.0629622)(18.15731,-1.8629622)(19.73731,-1.7429622)(19.73731,-0.9429621982574463)
\psline[linecolor=black, linewidth=0.04](19.67731,0.6370378)(19.73731,-1.0029622)
\psline[linecolor=black, linewidth=0.04](18.29731,-0.4629622)(18.53731,0.077037804)(16.99731,0.6570378)(16.99731,0.6570378)
\psline[linecolor=black, linewidth=0.04](17.95731,-0.5029622)(17.53731,0.017037801)(17.91731,0.2170378)
\psline[linecolor=black, linewidth=0.04](18.13731,0.3770378)(18.517311,0.6770378)
\psline[linecolor=black, linewidth=0.04](18.517311,1.2570378)(18.47731,1.6370378)
\psline[linecolor=black, linewidth=0.04](19.65731,1.2770379)(20.097311,1.6170378)
\psline[linecolor=black, linewidth=0.04](16.87731,1.2570378)(16.31731,1.6370378)(16.37731,1.5770378)
\rput[bl](18.31731,0.8570378){\Large{$i$}}
\rput[bl](16.81731,0.8570378){\Large{$i$}}
\end{pspicture}
}
\end{aligned}
\end{eqnarray*}
Consequently, 
 \begin{eqnarray*}
\begin{aligned}
\psscalebox{0.6 0.6} 
{
\begin{pspicture}(0,-1.3349096)(7.6574016,1.3349096)
\pscircle[linecolor=black, linewidth=0.04, dimen=outer](2.4174018,-0.4595769){0.34}
\pscircle[linecolor=black, linewidth=0.04, dimen=outer](1.0574018,-0.3395769){0.34}
\rput[bl](2.2774017,-0.5595769){\Large{$i$}}
\psbezier[linecolor=black, linewidth=0.04](1.0774018,-0.6395769)(1.0774018,-1.4395769)(2.4374018,-1.5595769)(2.4374018,-0.7595769119262695)
\rput[bl](3.7974017,-0.2395769){\LARGE{$=$}}
\rput[bl](0.83740175,-0.47957692){\Large{$n$}}
\psline[linecolor=black, linewidth=0.04](0.8574017,-0.15957691)(0.017401733,1.3204231)
\psline[linecolor=black, linewidth=0.04](1.2974018,-0.13957691)(1.9574018,1.3004231)(1.9774017,1.280423)
\psline[linecolor=black, linewidth=0.04](2.4974017,-0.17957692)(3.0574017,1.3004231)
\rput[bl](5.4574018,-0.21957691){\LARGE{$0$}}
\end{pspicture}
}
\end{aligned}
\end{eqnarray*}

Moreover, using Proposition~\ref{proponmandi} (i) twice, and then \eqref{ca2} and symmetry, we obtain
 \begin{eqnarray*}
\begin{aligned}
\psscalebox{0.6 0.6} 
{
\begin{pspicture}(0,-1.7476664)(21.077402,1.7476664)
\pscircle[linecolor=black, linewidth=0.04, dimen=outer](2.4174018,-0.35233366){0.34}
\pscircle[linecolor=black, linewidth=0.04, dimen=outer](1.0574018,-0.23233368){0.34}
\rput[bl](19.157402,-0.47233367){\Large{$i$}}
\psbezier[linecolor=black, linewidth=0.04](1.0774018,-0.5323337)(1.0774018,-1.3323337)(2.4374018,-1.4523337)(2.4374018,-0.6523336791992187)
\rput[bl](3.2774017,-0.17233367){\LARGE{$=$}}
\rput[bl](0.83740175,-0.37233368){\Large{$n$}}
\psline[linecolor=black, linewidth=0.04](0.8574017,-0.05233368)(0.017401733,1.4276663)
\psline[linecolor=black, linewidth=0.04](1.2974018,-0.03233368)(1.9574018,1.4076663)(1.9774017,1.3876663)
\rput[bl](2.2374017,-0.47233367){\Large{$e$}}
\rput[bl](10.197402,-0.21233368){\LARGE{$-$}}
\rput[bl](5.0774016,-0.25233367){\LARGE{$-$}}
\pscircle[linecolor=black, linewidth=0.04, dimen=outer](7.417402,-0.7723337){0.34}
\pscircle[linecolor=black, linewidth=0.04, dimen=outer](6.377402,0.24766631){0.34}
\pscircle[linecolor=black, linewidth=0.04, dimen=outer](9.077402,0.3276663){0.34}
\pscircle[linecolor=black, linewidth=0.04, dimen=outer](7.4974017,0.9876663){0.34}
\pscircle[linecolor=black, linewidth=0.04, dimen=outer](12.617402,-0.73233366){0.34}
\pscircle[linecolor=black, linewidth=0.04, dimen=outer](11.697402,-0.07233368){0.34}
\pscircle[linecolor=black, linewidth=0.04, dimen=outer](14.277402,0.36766633){0.34}
\pscircle[linecolor=black, linewidth=0.04, dimen=outer](12.657402,1.0276663){0.34}
\psbezier[linecolor=black, linewidth=0.04](7.437402,-1.0523337)(7.437402,-1.8523337)(8.797401,-1.9723337)(8.797401,-1.1723336791992187)
\psbezier[linecolor=black, linewidth=0.04](12.657402,-1.0323337)(12.657402,-1.8323337)(14.017402,-1.9523337)(14.017402,-1.1523336791992187)
\psline[linecolor=black, linewidth=0.04](7.5574017,0.64766634)(8.817402,-1.2323337)
\psline[linecolor=black, linewidth=0.04](6.4574018,-0.07233368)(7.2174015,-0.5723337)(7.2174015,-0.5923337)
\psline[linecolor=black, linewidth=0.04](6.237402,0.50766635)(5.5374017,1.6476663)(5.5174017,1.6476663)
\psline[linecolor=black, linewidth=0.04](7.4774017,1.2876663)(7.4774017,1.7476664)
\psline[linecolor=black, linewidth=0.04](7.6374016,-0.5323337)(8.017402,-0.31233367)
\psline[linecolor=black, linewidth=0.04](8.317402,-0.25233367)(9.077402,0.007666321)(9.077402,-0.03233368)
\psline[linecolor=black, linewidth=0.04](14.017402,-1.1723337)(12.677402,0.7276663)
\psline[linecolor=black, linewidth=0.04](12.597402,1.3476663)(11.657402,1.7276664)
\psline[linecolor=black, linewidth=0.04](12.397402,-0.5323337)(11.757401,-0.37233368)(11.777402,-0.41233367)
\psline[linecolor=black, linewidth=0.04](11.697402,0.24766631)(12.637402,0.46766633)
\psline[linecolor=black, linewidth=0.04](12.997402,0.5676663)(13.497402,0.70766634)(14.397402,1.6276664)(14.397402,1.6276664)
\psline[linecolor=black, linewidth=0.04](12.877401,-0.5123337)(13.2174015,-0.35233366)
\psline[linecolor=black, linewidth=0.04](13.517402,-0.23233368)(14.317402,0.047666322)
\rput[bl](7.1974015,-0.91233367){\Large{$m$}}
\rput[bl](16.497402,-0.23233368){\LARGE{$-2$}}
\rput[bl](12.397402,-0.85233366){\Large{$m$}}
\rput[bl](15.257401,-0.19233368){\LARGE{$=$}}
\rput[bl](14.057402,0.20766632){\Large{$e$}}
\rput[bl](12.457401,0.88766634){\Large{$i$}}
\rput[bl](11.497402,-0.19233368){\Large{$i$}}
\rput[bl](7.277402,0.8476663){\Large{$i$}}
\rput[bl](6.1574016,0.10766632){\Large{$i$}}
\rput[bl](8.857402,0.16766632){\Large{$e$}}
\rput[bl](17.757402,-0.49233368){\Large{$i$}}
\psbezier[linecolor=black, linewidth=0.04](17.97478,-0.66732454)(18.039244,-1.4647229)(19.404491,-1.4747412)(19.340025,-0.6773428278559623)
\pscircle[linecolor=black, linewidth=0.04, dimen=outer](19.337402,-0.35233366){0.34}
\pscircle[linecolor=black, linewidth=0.04, dimen=outer](17.977402,-0.35233366){0.34}
\psline[linecolor=black, linewidth=0.04](17.9374,-0.07233368)(17.837402,1.5876663)
\psline[linecolor=black, linewidth=0.04](19.317402,-0.07233368)(19.257402,1.5676663)
\end{pspicture}
}
\end{aligned}
\end{eqnarray*}
Now using the first equation of the direct sum property \eqref{directsumstring2} and the formula for $q$ in Proposition~\ref{Vselfdual}, we see that $n$ is equal to the right-hand side of the equation in our Proposition. 
\end{proof}
\begin{proposition}\label{AtoV}
For any composition algebra $A$ in $\CV$, the object $V$ as defined by \eqref{directsum}, equipped with the self-duality of Proposition~\ref{Vselfdual} and the wedge \eqref{wedgeforV}, is a vector product algebra in $\CV$. 
\end{proposition}
\begin{proof} It remains to prove the three axioms for a vertex algebra.
The first two axioms are taken care of by Proposition~\ref{cawedgeprop}: \eqref{stringasym} follows from \eqref{castringasym} while \eqref{stringcycsym} 
follows from \eqref{cathrowleg} and symmetry of the self duality.

To prove the remaining axiom \eqref{stringstrange}, we prove the equivalent form
in Corollary~\ref{stringswingcor4}.
Begin with $4$ times the first term on the left-hand side of \eqref{stringswingcor4}
where, for now, the strings are all labelled by $A$.
Substitute the formula \eqref{2wedge} for $\wedge$ in the two places to obtain
a sum of four terms, two of which are negative.
 \begin{eqnarray*}
\begin{aligned}
\psscalebox{0.6 0.6} 
{
\begin{pspicture}(0,-1.1025703)(18.56,1.1025703)
\psline[linecolor=black, linewidth=0.04](0.54,1.0825521)(1.02,-0.05744793)(1.64,1.0825521)(1.66,1.0825521)
\psline[linecolor=black, linewidth=0.04](2.5,1.0625521)(2.98,0.0025520707)(3.56,1.0825521)
\psbezier[linecolor=black, linewidth=0.04](1.04,-0.05744793)(1.04,-0.8574479)(3.0,-0.77744794)(3.0,0.022552070617675782)
\pscircle[linecolor=black, linewidth=0.04, dimen=outer](5.2,-0.15744793){0.34}
\pscircle[linecolor=black, linewidth=0.04, dimen=outer](6.58,-0.15744793){0.34}
\pscircle[linecolor=black, linewidth=0.04, dimen=outer](8.58,-0.13744792){0.34}
\pscircle[linecolor=black, linewidth=0.04, dimen=outer](9.98,-0.13744792){0.34}
\psbezier[linecolor=black, linewidth=0.04](5.2,-0.49744794)(5.2,-1.2974479)(6.6,-1.237448)(6.6,-0.43744792938232424)
\psline[linecolor=black, linewidth=0.04](5.0,0.08255207)(4.5,1.0625521)
\psline[linecolor=black, linewidth=0.04](5.42,0.08255207)(5.74,1.0425521)
\psline[linecolor=black, linewidth=0.04](6.76,0.12255207)(7.26,1.0825521)
\psbezier[linecolor=black, linewidth=0.04](8.62,-0.47744793)(8.62,-1.2774479)(10.0,-1.237448)(10.0,-0.43744792938232424)
\psline[linecolor=black, linewidth=0.04](8.42,0.14255208)(8.1,0.5625521)(9.0,1.0625521)
\psline[linecolor=black, linewidth=0.04](7.82,1.0625521)(8.36,0.84255207)
\psline[linecolor=black, linewidth=0.04](8.62,0.72255206)(9.12,0.5025521)(8.76,0.16255207)
\psline[linecolor=black, linewidth=0.04](9.72,0.04255207)(9.3,0.54255205)(10.38,1.0425521)
\psline[linecolor=black, linewidth=0.04](9.34,1.0425521)(9.72,0.8825521)
\psline[linecolor=black, linewidth=0.04](10.02,0.74255204)(10.42,0.5625521)(10.2,0.08255207)
\pscircle[linecolor=black, linewidth=0.04, dimen=outer](12.02,-0.15744793){0.34}
\pscircle[linecolor=black, linewidth=0.04, dimen=outer](13.42,-0.15744793){0.34}
\psbezier[linecolor=black, linewidth=0.04](12.06,-0.49744794)(12.06,-1.2974479)(13.44,-1.257448)(13.44,-0.4574479293823242)
\psline[linecolor=black, linewidth=0.04](11.86,0.12255207)(11.54,0.54255205)(12.44,1.0425521)
\psline[linecolor=black, linewidth=0.04](11.26,1.0425521)(11.8,0.8225521)
\psline[linecolor=black, linewidth=0.04](12.06,0.7025521)(12.56,0.48255208)(12.2,0.14255208)
\pscircle[linecolor=black, linewidth=0.04, dimen=outer](15.2,-0.15744793){0.34}
\pscircle[linecolor=black, linewidth=0.04, dimen=outer](16.6,-0.15744793){0.34}
\psbezier[linecolor=black, linewidth=0.04](15.24,-0.49744794)(15.24,-1.2974479)(16.62,-1.257448)(16.62,-0.4574479293823242)
\psline[linecolor=black, linewidth=0.04](16.34,0.022552071)(15.92,0.5225521)(17.0,1.022552)
\psline[linecolor=black, linewidth=0.04](15.96,1.022552)(16.34,0.86255205)
\psline[linecolor=black, linewidth=0.04](16.64,0.72255206)(17.04,0.54255205)(16.82,0.06255207)
\psline[linecolor=black, linewidth=0.04](13.24,0.06255207)(12.76,1.022552)
\psline[linecolor=black, linewidth=0.04](13.58,0.14255208)(14.0,1.022552)
\psline[linecolor=black, linewidth=0.04](14.96,0.06255207)(14.6,1.022552)
\psline[linecolor=black, linewidth=0.04](15.4,0.10255207)(15.54,1.022552)
\psline[linecolor=black, linewidth=0.04](6.38,0.08255207)(6.14,1.0625521)
\rput[bl](0.0,0.08255207){\LARGE{$4$}}
\rput[bl](7.48,-0.13744792){\LARGE{$+$}}
\rput[bl](14.06,-0.17744793){\LARGE{$-$}}
\rput[bl](16.4,-0.25744793){\Large{$m$}}
\rput[bl](10.92,-0.17744793){\LARGE{$-$}}
\rput[bl](3.84,-0.17744793){\LARGE{$=$}}
\rput[bl](14.96,-0.29744792){\Large{$m$}}
\rput[bl](13.2,-0.27744794){\Large{$m$}}
\rput[bl](11.82,-0.27744794){\Large{$m$}}
\rput[bl](9.76,-0.25744793){\Large{$m$}}
\rput[bl](8.36,-0.23744793){\Large{$m$}}
\rput[bl](6.38,-0.25744793){\Large{$m$}}
\rput[bl](4.98,-0.23744793){\Large{$m$}}
\end{pspicture}
}
\end{aligned}
\end{eqnarray*}
Apply composition algebra \eqref{ca3} to each of the positive terms on the right-hand side
to obtain
 \begin{eqnarray*}
\begin{aligned}
\psscalebox{0.6 0.6} 
{
\begin{pspicture}(0,-1.2984811)(19.0,1.2984811)
\pscircle[linecolor=black, linewidth=0.04, dimen=outer](5.64,-0.35335878){0.34}
\pscircle[linecolor=black, linewidth=0.04, dimen=outer](7.02,-0.35335878){0.34}
\pscircle[linecolor=black, linewidth=0.04, dimen=outer](9.02,-0.3333588){0.34}
\pscircle[linecolor=black, linewidth=0.04, dimen=outer](10.42,-0.3333588){0.34}
\psbezier[linecolor=black, linewidth=0.04](5.64,-0.6933588)(5.64,-1.4933587)(7.04,-1.4333588)(7.04,-0.6333587837219238)
\psline[linecolor=black, linewidth=0.04](5.44,-0.11335878)(4.88,0.9666412)
\psbezier[linecolor=black, linewidth=0.04](9.06,-0.6733588)(9.06,-1.4733588)(10.44,-1.4333588)(10.44,-0.6333587837219238)
\psline[linecolor=black, linewidth=0.04](10.198889,-0.15335879)(9.74,0.42664123)(10.92,1.0066413)
\psline[linecolor=black, linewidth=0.04](10.379197,0.558573)(10.892942,0.32147157)(10.586707,-0.108668245)
\pscircle[linecolor=black, linewidth=0.04, dimen=outer](12.46,-0.35335878){0.34}
\pscircle[linecolor=black, linewidth=0.04, dimen=outer](13.86,-0.35335878){0.34}
\psbezier[linecolor=black, linewidth=0.04](12.5,-0.6933588)(12.5,-1.4933587)(13.88,-1.4533588)(13.88,-0.6533587837219238)
\psline[linecolor=black, linewidth=0.04](12.3,-0.07335878)(11.98,0.3466412)(12.88,0.84664124)
\psline[linecolor=black, linewidth=0.04](11.7,0.84664124)(12.24,0.6266412)
\psline[linecolor=black, linewidth=0.04](12.5,0.5066412)(13.0,0.2866412)(12.64,-0.053358782)
\pscircle[linecolor=black, linewidth=0.04, dimen=outer](15.64,-0.35335878){0.34}
\pscircle[linecolor=black, linewidth=0.04, dimen=outer](17.04,-0.35335878){0.34}
\psbezier[linecolor=black, linewidth=0.04](15.68,-0.6933588)(15.68,-1.4933587)(17.06,-1.4533588)(17.06,-0.6533587837219238)
\psline[linecolor=black, linewidth=0.04](16.78,-0.17335878)(16.36,0.3266412)(17.44,0.8266412)
\psline[linecolor=black, linewidth=0.04](16.4,0.8266412)(16.78,0.66664124)
\psline[linecolor=black, linewidth=0.04](17.08,0.5266412)(17.48,0.3466412)(17.26,-0.13335878)
\psline[linecolor=black, linewidth=0.04](13.68,-0.13335878)(13.2,0.8266412)
\psline[linecolor=black, linewidth=0.04](14.02,-0.053358782)(14.44,0.8266412)
\psline[linecolor=black, linewidth=0.04](15.4,-0.13335878)(15.04,0.8266412)
\psline[linecolor=black, linewidth=0.04](15.84,-0.093358785)(15.98,0.8266412)
\rput[bl](7.92,-0.3333588){\LARGE{$-$}}
\rput[bl](14.5,-0.3733588){\LARGE{$-$}}
\rput[bl](16.84,-0.45335877){\Large{$m$}}
\rput[bl](11.36,-0.3733588){\LARGE{$-$}}
\rput[bl](0.0,-0.31335878){\LARGE{$=$}}
\rput[bl](15.4,-0.4933588){\Large{$m$}}
\rput[bl](13.64,-0.47335878){\Large{$m$}}
\rput[bl](12.26,-0.47335878){\Large{$m$}}
\rput[bl](10.2,-0.45335877){\Large{$m$}}
\rput[bl](8.8,-0.4333588){\Large{$m$}}
\rput[bl](6.82,-0.45335877){\Large{$m$}}
\rput[bl](5.42,-0.4333588){\Large{$m$}}
\rput[bl](1.1,-0.2733588){\LARGE{$4$}}
\psbezier[linecolor=black, linewidth=0.04](1.9,-0.59335876)(1.9,-1.3933588)(2.78,-1.3933588)(2.78,-0.5933587837219239)
\rput[bl](4.4,-0.31335878){\LARGE{$-$}}
\psbezier[linecolor=black, linewidth=0.04](3.24,-0.59335876)(3.24,-1.3933588)(4.12,-1.3933588)(4.12,-0.5933587837219239)
\psline[linecolor=black, linewidth=0.04](2.8,-0.59335876)(3.48,0.8666412)
\psline[linecolor=black, linewidth=0.04](3.24,-0.6133588)(3.12,-0.17335878)
\psline[linecolor=black, linewidth=0.04](2.96,0.16664122)(2.6,0.8666412)
\psline[linecolor=black, linewidth=0.04](4.12,-0.6133588)(4.16,0.84664124)
\psline[linecolor=black, linewidth=0.04](1.9,-0.6133588)(1.92,0.84664124)
\psline[linecolor=black, linewidth=0.04](7.18,-0.07335878)(7.5,0.36664122)(6.76,0.72664124)(6.74,0.72664124)
\psline[linecolor=black, linewidth=0.04](6.84,-0.093358785)(6.36,0.5066412)
\psline[linecolor=black, linewidth=0.04](6.38,0.46664122)(6.82,1.0866412)
\psline[linecolor=black, linewidth=0.04](6.48,0.78664124)(6.16,1.0466412)
\psline[linecolor=black, linewidth=0.04](5.84,-0.11335878)(6.4,0.22664122)
\psline[linecolor=black, linewidth=0.04](6.64,0.30664122)(6.94,0.5066412)
\psline[linecolor=black, linewidth=0.04](7.16,0.66664124)(7.62,1.0466412)
\psline[linecolor=black, linewidth=0.04](10.02,0.66664124)(8.86,1.0666412)
\psline[linecolor=black, linewidth=0.04](8.9,-0.053358782)(9.2,0.8066412)
\psline[linecolor=black, linewidth=0.04](9.32,0.9866412)(9.54,1.2866412)
\psline[linecolor=black, linewidth=0.04](9.22,-0.07335878)(9.62,0.72664124)
\psline[linecolor=black, linewidth=0.04](9.72,0.9266412)(9.98,1.2466412)
\end{pspicture}
}
\end{aligned}
\end{eqnarray*}
The second term on the left-hand side of \eqref{stringswingcor4} is obtained from the first by
composing with the braiding $c_{A,A\ox A\ox A}$. So we also obtain an expression for $4$ times this second term as a combination of five terms where the first coefficient is $+4$
and the other four are $-1$. Now add the two five term expressions.
The positive terms in each expression are the same so we obtain an $8$ as coefficient.
The negative terms factorize. 
(In sorting out the string diagrams for this, the reader must note that $\bigcdot$ and $m$ do not see the difference between preceding with the braiding or with its inverse.)  
If we precede this expression by $i\ox i\ox i \ox i$ and again denote the left-hand side of Proposition~\ref{m+switchm} by $n$, we find that the left-hand side of Corollary~\ref{stringswingcor4} is equal to
 \begin{eqnarray*}
\begin{aligned}
\psscalebox{0.6 0.6} 
{
\begin{pspicture}(0,-1.6089694)(13.78,1.6089694)
\pscircle[linecolor=black, linewidth=0.04, dimen=outer](7.83,-0.38633522){0.32}
\pscircle[linecolor=black, linewidth=0.04, dimen=outer](6.34,-0.4363352){0.34}
\pscircle[linecolor=black, linewidth=0.04, dimen=outer](11.94,-0.4363352){0.34}
\pscircle[linecolor=black, linewidth=0.04, dimen=outer](10.5,-0.45633522){0.34}
\psbezier[linecolor=black, linewidth=0.04](6.38,-0.7363352)(6.38,-1.5363352)(7.78,-1.4763352)(7.78,-0.6763352203369141)
\psbezier[linecolor=black, linewidth=0.04](10.52,-0.77633524)(10.52,-1.5763352)(11.9,-1.5363352)(11.9,-0.736335220336914)
\rput[bl](4.46,-0.21633522){\LARGE{$-$}}
\psbezier[linecolor=black, linewidth=0.04](1.22,-0.9763352)(1.22,-1.7763352)(2.1,-1.7763352)(2.1,-0.9763352203369141)
\pscircle[linecolor=black, linewidth=0.04, dimen=outer](3.54,0.3836648){0.34}
\pscircle[linecolor=black, linewidth=0.04, dimen=outer](2.8,0.3836648){0.34}
\pscircle[linecolor=black, linewidth=0.04, dimen=outer](2.02,0.40366477){0.34}
\pscircle[linecolor=black, linewidth=0.04, dimen=outer](1.28,0.3836648){0.34}
\psline[linecolor=black, linewidth=0.04](1.26,0.7036648)(1.28,1.5836648)
\psline[linecolor=black, linewidth=0.04](2.04,0.72366476)(2.04,1.5836648)
\psline[linecolor=black, linewidth=0.04](2.8,0.7036648)(2.78,1.5636648)
\psline[linecolor=black, linewidth=0.04](3.58,0.6836648)(3.58,1.5836648)
\rput[bl](8.78,-0.11633522){\LARGE{$-$}}
\rput[bl](3.42,0.22366478){\Large{$i$}}
\rput[bl](7.64,-0.4963352){\Large{$n$}}
\rput[bl](0.0,-0.2763352){\LARGE{$8$}}
\rput[bl](2.66,0.24366479){\Large{$i$}}
\rput[bl](1.86,0.26366478){\Large{$i$}}
\rput[bl](1.12,0.24366479){\Large{$i$}}
\rput[bl](6.16,-0.5563352){\Large{$n$}}
\rput[bl](10.32,-0.5563352){\Large{$n$}}
\rput[bl](11.74,-0.53633523){\Large{$n$}}
\psline[linecolor=black, linewidth=0.04](7.66,-0.13633522)(7.3,1.5436648)
\psline[linecolor=black, linewidth=0.04](7.98,-0.11633522)(7.98,1.5436648)
\psline[linecolor=black, linewidth=0.04](10.32,-0.21633522)(10.04,1.5836648)
\psline[linecolor=black, linewidth=0.04](10.66,-0.15633522)(11.02,1.5836648)
\psline[linecolor=black, linewidth=0.04](11.76,-0.15633522)(11.48,1.5636648)
\psline[linecolor=black, linewidth=0.04](12.2,-0.21633522)(12.7,1.6036648)
\psline[linecolor=black, linewidth=0.04](6.14,-0.21633522)(5.86,1.5636648)
\psline[linecolor=black, linewidth=0.04](6.5,-0.17633522)(7.36,0.54366475)
\psline[linecolor=black, linewidth=0.04](7.62,0.6436648)(7.88,0.84366477)
\psline[linecolor=black, linewidth=0.04](8.1,0.90366477)(8.78,1.5436648)
\psbezier[linecolor=black, linewidth=0.04](2.66,-0.9963352)(2.66,-1.7963352)(3.54,-1.7763352)(3.54,-0.9763352203369141)
\psline[linecolor=black, linewidth=0.04](2.1,-0.9963352)(2.8,0.10366478)
\psline[linecolor=black, linewidth=0.04](1.22,-1.0163352)(1.22,0.10366478)
\psline[linecolor=black, linewidth=0.04](3.54,-0.9963352)(3.54,0.08366478)(3.54,0.04366478)
\psline[linecolor=black, linewidth=0.04](2.68,-1.0563352)(2.5,-0.65633523)
\psline[linecolor=black, linewidth=0.04](2.32,-0.45633522)(2.0,0.10366478)
\end{pspicture}
}
\end{aligned}
\end{eqnarray*}
from which the result follows by using the formula for $n$ in Proposition~\ref{m+switchm}.  
\end{proof}
\begin{corollary}
$(d_A-1)(d_A-2)(d_A-4)(d_A-8)=0$
\end{corollary}
\begin{proof}
As the dimension of $X\oplus Y$ is $d_X+d_Y$, we see from \eqref{directsum} that $d_A=1+d_V$.
The result now follows from Proposition~\ref{AtoV} and Theorem~\ref{scarcityvpa}.
\end{proof}

Let $\mathbf{CA}$ denote the category of composition algebras; morphisms $f: A\to B$ are those which preserve the operations: 
$$(A\ox A\xra{f\ox f}B\ox B\xra{\bigcdot}I) = (A\ox A\xra{\bigcdot}I) \ ,$$
$$(A\ox A\xra{f\ox f}B\ox B\xra{m}B) = (A\ox A\xra{\bigcdot}A\xra{f} B) \ ,$$ 
$$(I\xra{e}A\xra{f}B) = (I\xra{e}B) \ .$$

Let $\mathbf{VPA}$ denote the category of composition algebras; morphisms $h: V\to W$ are those which preserve the operation $\bigcdot$ as above and $\wedge$ in the usual sense.
\begin{theorem}
Let $\CV$ be a braided monoidal additive category with finite direct sums and splitting of idempotents. 
Assume also that multiplication by $2$ is invertible in each hom abelian group $\CV(X,Y)$.
Then the functor $\Phi : \mathbf{CA}\to \mathbf{VPA}$, taking each composition algebra $A$ to the vpa $V$ of Proposition~\ref{AtoV} and each morphism to its restriction,
is an equivalence of categories.  
\end{theorem}
\begin{proof}
That $\Phi$ is fully faithful follows from \eqref{directsum}, \eqref{2wedge} and \eqref{m+switchm} which show how $\wedge$ and $m$ can be defined in terms of each other.

It remains to show that $\Phi$ is essentially surjective on objects.
Take a vpa $V$. Using the direct sums in $\CV$, there exists an $A$ as
defined by the direct sum diagram \eqref{directsum}.
A symmetric self-duality for $A$ is defined by
 \begin{eqnarray}\label{Aselfdual}
\begin{aligned}
\psscalebox{0.6 0.6} 
{
\begin{pspicture}(0,-1.0184594)(11.46,1.0184594)
\psbezier[linecolor=black, linewidth=0.04](0.56,-0.28252193)(0.56,-1.0825219)(2.04,-1.1025219)(2.04,-0.30252193450927733)
\psline[linecolor=black, linewidth=0.04](2.04,-0.36252195)(2.06,0.9774781)
\psline[linecolor=black, linewidth=0.04](0.56,-0.30252194)(0.56,0.95747805)
\psbezier[linecolor=black, linewidth=0.04](4.08,-0.18252194)(4.08,-0.98252195)(5.56,-1.002522)(5.56,-0.20252193450927736)
\pscircle[linecolor=black, linewidth=0.04, dimen=outer](9.24,-0.10252193){0.32}
\pscircle[linecolor=black, linewidth=0.04, dimen=outer](7.88,-0.10252193){0.32}
\pscircle[linecolor=black, linewidth=0.04, dimen=outer](5.54,0.09747806){0.32}
\pscircle[linecolor=black, linewidth=0.04, dimen=outer](4.06,0.09747806){0.32}
\psline[linecolor=black, linewidth=0.04](4.04,0.39747807)(4.02,0.99747807)
\psline[linecolor=black, linewidth=0.04](5.54,0.37747806)(5.52,0.99747807)
\psline[linecolor=black, linewidth=0.04](7.86,0.19747807)(7.82,1.0174781)
\psline[linecolor=black, linewidth=0.04](9.24,0.17747806)(9.2,0.99747807)
\rput[bl](5.34,-0.062521935){\Large{$q$}}
\rput[bl](7.68,-0.26252192){\Large{$p$}}
\rput[bl](5.72,0.57747805){\Large{$A$}}
\rput[bl](5.5,-0.9425219){\Large{$V$}}
\rput[bl](6.62,-0.0025219345){\LARGE{$+$}}
\rput[bl](2.94,0.017478066){\LARGE{$=$}}
\rput[bl](3.9,-0.062521935){\Large{$q$}}
\rput[bl](9.06,-0.24252194){\Large{$p$}}
\rput[bl](4.14,0.57747805){\Large{$A$}}
\rput[bl](7.94,0.6574781){\Large{$A$}}
\rput[bl](9.4,0.6774781){\Large{$A$}}
\rput[bl](2.16,0.37747806){\Large{$A$}}
\rput[bl](0.0,0.37747806){\Large{$A$}}
\end{pspicture}
}
\end{aligned}
\end{eqnarray}
while we have the identity $e$ as part of diagram \eqref{directsum}.
We define the multiplication $m: A\ox A\to A$ by insisting that 
$e$ is an identity for that multiplication together with the following two equations.
 \begin{eqnarray}\label{definem}
\begin{aligned}
\psscalebox{0.6 0.6} 
{
\begin{pspicture}(0,-2.1299999)(18.02,2.1299999)
\psbezier[linecolor=black, linewidth=0.04](4.34,0.24999996)(4.34,-0.55)(5.82,-0.57000005)(5.82,0.22999996185302735)
\psline[linecolor=black, linewidth=0.04](5.82,0.16999996)(5.84,2.09)
\psline[linecolor=black, linewidth=0.04](4.34,0.22999996)(4.34,2.11)
\pscircle[linecolor=black, linewidth=0.04, dimen=outer](11.1,0.83){0.32}
\pscircle[linecolor=black, linewidth=0.04, dimen=outer](1.1,-1.09){0.32}
\pscircle[linecolor=black, linewidth=0.04, dimen=outer](1.94,0.86999995){0.32}
\pscircle[linecolor=black, linewidth=0.04, dimen=outer](0.32,0.89){0.32}
\rput[bl](4.46,1.63){\Large{$V$}}
\rput[bl](13.64,0.14999996){\LARGE{$=$}}
\rput[bl](11.82,-1.11){\Large{$q$}}
\rput[bl](0.94,-1.23){\Large{$p$}}
\pscircle[linecolor=black, linewidth=0.04, dimen=outer](1.12,0.12999997){0.32}
\psline[linecolor=black, linewidth=0.04](0.32,1.1899999)(0.32,2.1299999)
\psline[linecolor=black, linewidth=0.04](1.94,1.15)(1.94,2.11)
\psline[linecolor=black, linewidth=0.04](0.42,0.59)(0.84,0.26999995)
\psline[linecolor=black, linewidth=0.04](1.9,0.59)(1.4,0.24999996)
\psline[linecolor=black, linewidth=0.04](1.08,-0.17000003)(1.1,-0.83000004)
\rput[bl](3.62,0.68999994){\LARGE{$-$}}
\rput[bl](11.68,-0.07000004){\Large{$m$}}
\rput[bl](12.56,0.66999996){\Large{$i$}}
\rput[bl](0.2,0.74999994){\Large{$i$}}
\rput[bl](1.8,0.72999996){\Large{$i$}}
\rput[bl](2.56,0.59){\LARGE{$=$}}
\rput[bl](0.88,-0.010000038){\Large{$m$}}
\pscircle[linecolor=black, linewidth=0.04, dimen=outer](12.7,0.83){0.32}
\pscircle[linecolor=black, linewidth=0.04, dimen=outer](11.92,0.06999996){0.32}
\pscircle[linecolor=black, linewidth=0.04, dimen=outer](11.98,-0.95000005){0.32}
\psline[linecolor=black, linewidth=0.04](14.4,2.09)(15.2,0.32999995)(15.2,-2.13)
\psline[linecolor=black, linewidth=0.04](15.22,0.30999997)(15.92,2.11)
\psline[linecolor=black, linewidth=0.04](11.06,1.11)(11.04,2.09)
\psline[linecolor=black, linewidth=0.04](12.72,1.11)(12.7,2.09)
\psline[linecolor=black, linewidth=0.04](11.18,0.53)(11.72,0.28999996)
\psline[linecolor=black, linewidth=0.04](12.66,0.53)(12.16,0.28999996)
\psline[linecolor=black, linewidth=0.04](11.96,-0.21000004)(11.98,-0.69000006)
\psline[linecolor=black, linewidth=0.04](12.0,-1.27)(12.0,-2.07)
\rput[bl](11.02,0.66999996){\Large{$i$}}
\rput[bl](14.04,1.3499999){\Large{$V$}}
\rput[bl](15.96,1.37){\Large{$V$}}
\end{pspicture}
}
\end{aligned}
\end{eqnarray}
Axioms \eqref{ca1}, \eqref{ca0}, \eqref{ca2} are then obvious.
In axiom \eqref{ca3}, we replace the four counits for $A\dashv A$ occurring in the equation by the right-hand side of \eqref{Aselfdual}.
Then it suffices to check the result for the sixteen cases obtained by attaching
either of the direct sum injections $e$ or $i$ to the four input strings.
The only case that needs attention is when all four input strings have $i$ attached;
this produces the following condition.
 \begin{eqnarray}\label{ca3restricted}
\begin{aligned}
\psscalebox{0.5 0.5} 
{
\begin{pspicture}(0,-1.7784991)(21.699995,1.7784991)
\psbezier[linecolor=black, linewidth=0.04](0.52,-1.0059427)(0.52,-1.8059428)(2.0,-1.8259428)(2.0,-1.0259427261352538)
\rput[bl](16.14,-0.045942727){\LARGE{$=$}}
\psbezier[linecolor=black, linewidth=0.04](20.18,-1.1459427)(20.18,-1.9459428)(21.66,-1.9659427)(21.66,-1.165942726135254)
\psbezier[linecolor=black, linewidth=0.04](18.36,-1.1059427)(18.36,-1.9059427)(19.84,-1.9259428)(19.84,-1.125942726135254)
\psbezier[linecolor=black, linewidth=0.04](14.3,-1.0859427)(14.3,-1.8859427)(15.78,-1.9059427)(15.78,-1.105942726135254)
\psbezier[linecolor=black, linewidth=0.04](12.4,-1.0659428)(12.4,-1.8659427)(13.88,-1.8859427)(13.88,-1.0859427261352539)
\psbezier[linecolor=black, linewidth=0.04](9.1,-1.0459428)(9.1,-1.8459427)(10.58,-1.8659427)(10.58,-1.0659427261352539)
\psbezier[linecolor=black, linewidth=0.04](5.8,-1.0059427)(5.8,-1.8059428)(7.28,-1.8259428)(7.28,-1.0259427261352538)
\psbezier[linecolor=black, linewidth=0.04](3.94,-0.9859427)(3.94,-1.7859427)(5.42,-1.8059428)(5.42,-1.0059427261352538)
\psline[linecolor=black, linewidth=0.04](18.36,-1.1259427)(18.34,1.7340573)
\psline[linecolor=black, linewidth=0.04](19.86,-1.1859428)(19.76,1.6940572)
\psline[linecolor=black, linewidth=0.04](20.18,-1.1459427)(20.08,1.7140573)
\psline[linecolor=black, linewidth=0.04](21.68,-1.2059427)(21.62,1.7140573)
\rput[bl](17.48,-0.045942727){\LARGE{$2$}}
\psline[linecolor=black, linewidth=0.04](0.52,-1.0259427)(0.02,1.7540573)(0.02,1.6940572)
\psline[linecolor=black, linewidth=0.04](2.0,-1.0659428)(2.18,0.11405727)(0.72,1.7540573)
\psline[linecolor=black, linewidth=0.04](2.2,0.11405727)(2.58,1.7740573)
\psline[linecolor=black, linewidth=0.04](0.36,-0.12594272)(1.2,0.8140573)
\psline[linecolor=black, linewidth=0.04](1.56,1.1140573)(2.08,1.7540573)
\psline[linecolor=black, linewidth=0.04](5.8,-1.0259427)(4.86,1.7340573)
\psline[linecolor=black, linewidth=0.04](5.42,-1.0459428)(5.48,-0.58594275)
\psline[linecolor=black, linewidth=0.04](5.62,-0.14594273)(6.22,1.7340573)
\psline[linecolor=black, linewidth=0.04](7.28,-1.0459428)(7.26,1.7140573)
\psline[linecolor=black, linewidth=0.04](3.94,-1.0259427)(3.98,1.7340573)
\psline[linecolor=black, linewidth=0.04](9.1,-1.0659428)(8.44,1.7340573)
\psline[linecolor=black, linewidth=0.04](10.6,-1.1059427)(10.58,1.6940572)
\psline[linecolor=black, linewidth=0.04](10.6,-0.005942726)(9.46,1.6940572)
\psline[linecolor=black, linewidth=0.04](8.9,-0.24594273)(10.0,0.5140573)
\psline[linecolor=black, linewidth=0.04](10.26,0.71405727)(10.46,0.8740573)
\psline[linecolor=black, linewidth=0.04](10.74,0.9940573)(11.52,1.6940572)
\psline[linecolor=black, linewidth=0.04](14.32,-1.1259427)(13.06,1.6940572)
\psline[linecolor=black, linewidth=0.04](15.78,-1.1259427)(13.98,1.6940572)
\psline[linecolor=black, linewidth=0.04](12.4,-1.0859427)(12.4,1.6940572)
\psline[linecolor=black, linewidth=0.04](13.9,-1.1059427)(14.0,-0.7459427)
\psline[linecolor=black, linewidth=0.04](14.14,-0.48594272)(14.54,0.41405728)
\psline[linecolor=black, linewidth=0.04](14.8,0.7540573)(15.22,1.6540573)
\rput[bl](11.3,-0.005942726){\LARGE{$+$}}
\rput[bl](7.82,-0.005942726){\LARGE{$+$}}
\rput[bl](2.78,-0.005942726){\LARGE{$+$}}
\end{pspicture}
}
\end{aligned}
\end{eqnarray}
However, if you take the vpa axiom \eqref{stringstrange}, move the negative terms to the
left-hand side, drag the bottom string in each term up to the right, apply the braiding 
to the top middle two strings of each term, and manipulate a little using the earlier vpa axioms and \eqref{stringswing}, we obtain \eqref{ca3restricted}. 
Finally, observe that $\Phi(A) \cong V$.
\end{proof}  
 
\begin{center}
--------------------------------------------------------
\end{center}

\appendix

\end{document}